\documentclass{amsart}

\usepackage[latin1]{inputenc}
\usepackage{harvard}
\usepackage{graphicx}
\usepackage{psfrag}
\usepackage{amsthm,amsfonts,amssymb,amscd,amsmath}

\newtheorem{theorem}{Theorem}[section]
\newtheorem*{mainthm}{Main Theorem}

\newtheorem{lemma}[theorem]{Lemma}
\newtheorem{oldthm}{Theorem}

\newtheorem{prop}[theorem]{Proposition}

\newtheorem{remark}[theorem]{Remark}

\newtheorem{definition}[theorem]{Definition}

\newcommand{\field}[1]{\mathbb{#1}}

\def\C{\mathbb{C}}
\def\R{\mathbb{R}}

\def\eps{\epsilon}
\def\la{\lambda}

\def\fC{\field{C}}

\def\fI{\field{I}}

\def\fR{\field{R}}

\def\cA{\mathcal{A}}
\def\cB{\mathcal{B}}
\def\cC{\mathcal{C}}
\def\cD{\mathcal{D}}

\def\cF{\mathcal{F}}

\def\cI{\mathcal{I}}

\def\cL{\mathcal{L}}
\def\cN{\mathcal{N}}

\def\cR{\mathcal{R}}

\def\cP{\mathcal{P}}
\def\cS{\mathcal{S}}

\def\cW{\mathcal{W}}

\def\cO{\mathcal{O}}

\def\aF{\cF_{\varrho_\beta}^{\beta,\rho_\beta}(s^*)}
\def\aaF{\cF_{\varrho}^{\beta,\rho}(s^*)}
\def\bF{\cF_{\varrho_\beta}^{\beta,\rho_\beta}(s_\beta)}
\def\Bstar{\cB_\varrho(s^*)}

\def\bB{{\bf B}}

\date{2014-12-01}

\begin{document}
\title[Spectral properties of renormalization for area-preserving maps]{ Spectral properties of renormalization for area-preserving maps  }
\author{Denis Gaidashev}
\address{
Department of Mathematics, Uppsala University, Uppsala, Sweden\\
{\tt gaidash@math.uu.se}}

\author{Tomas Johnson}
\address{Fraunhofer-Chalmers Research Centre for Industrial Mathematics
Chalmers University of Technology, 
SE-412 88 Gothenburg, Sweden,
{\tt tomas.johnson@fcc.chalmers.se} }

\begin{abstract}
Area-preserving maps have been observed to undergo a universal period-doubling cascade, analogous to the famous Feigenbaum-Coullet-Tresser period doubling cascade in one-dimensional dynamics. A renormalization approach has been used by Eckmann, Koch and Wittwer in a computer-assisted proof of existence of a conservative renormalization fixed point. 

Furthermore, it has been shown by Gaidashev, Johnson and Martens that {\it infinitely renormalizable maps} in a neighborhood of this fixed point admit invariant Cantor sets with vanishing Lyapunov exponents on which dynamics for any two maps is smoothly conjugate.

This rigidity is a consequence of an interplay between the decay of geometry and the convergence rate of renormalization towards  the fixed point.  

In this paper we prove a  result which is crucial for a demonstration of rigidity: that an upper bound on this convergence rate of renormalizations  of infinitely renormalizable maps is sufficiently small.
\end{abstract}

\maketitle

\setcounter{page}{1}

\tableofcontents

\bigskip \section*{Introduction}

Following the pioneering discovery of the Feigenbaum-Coullet-Tresser period doubling universality in unimodal maps \cite{Fei1}, \cite{Fei2}, \cite{TC}, universality ---  independence of the quantifiers of the geometry of orbits and bifurcation cascades in families of maps of the choice of a particular family --- has been demonstrated to be a rather generic phenomenon in dynamics. 

Universality problems are typically approached via {\it renormalization}. In a renormalization setting one introduces  a {\it renormalization} operator on a functional space, and demonstrates that this operator has a {\it hyperbolic fixed point}.  This approach has been very successful in one-dimensional dynamics, and has led to explanation of universality in unimodal maps \cite{Eps2}, \cite{Lyu}, \cite{Mar}, critical circle maps \cite{dF1,dF2,Ya1,Ya2} and holomorphic maps with a  Siegel disk \cite{McM,Ya3,GaiYa}. There is, however, at present  no complete understanding of universality in conservative systems, other than in the  case of the universality for systems ``near integrability'' \cite{AK,AKW,Koch1,Koch2,Koch3,Gai1,Kocic,KLDM}. 

Period-doubling renormalization for two-dimensional maps has been extensively studied in \cite{CEK1,dCLM,LM}.   Specifically, the authors of \cite{dCLM} have considered strongly dissipative H\'enon-like maps of the form  
\begin{equation}\label{Henon-like}
F(x,y)=(f(x)-\epsilon(x,y),x),
\end{equation}
where $f(x)$ is a unimodal map (subject to some regularity conditions), and $\epsilon$ is small. Whenever the one-dimensional map $f$ is {\it renormalizable}, one can define a renormalization of $F$, following \cite{dCLM}, as 
$$R_{dCLM}[F]=H^{-1} \circ F \circ F \arrowvert_U \circ H,$$
where $U$ is an appropriate neighborhood of the critical value $v=(f(0),0)$, and $H$ is an explicit non-linear change of coordinates. \cite{dCLM} demonstrates that  the degenerate map $F_*(x,y)=(f_*(x),x)$, where $f_*$ is the Feigenbaum-Collet-Tresser fixed point of one-dimensional renormalization, is a hyperbolic fixed point of $R_{dCLM}$. Furthermore, according to \cite{dCLM}, for any infinitely-renormalizable map of the form $(\ref{Henon-like})$, there exists a hierarchical family of ``pieces'' $\{B_\sigma^n\}$,  organized by inclusion in a dyadic tree, such that the set
$$\cC_F=\bigcap_n \bigcup_\sigma B_\sigma^n$$
is an attracting Cantor set on which  $F$ acts as an {\it adding machine}. Compared to the Feigenbaum-Collet-Tresser one-dimensional  renormalization, the new striking feature of the two dimensional renormalization for highly dissipative maps $(\ref{Henon-like})$, is that {\it the restriction of the dynamics to this Cantor set  is not rigid}. Indeed, if the average Jacobians of $F$ and $G$ are different, for example, $b_F<b_G$, then the conjugacy $F \arrowvert_{\cC_F}  \, { \approx \atop { h} }    \, G \arrowvert_{\cC_G}$ is not smooth, rather it is at best
a  H\"older continuous function with a definite upper bound on the H\"older exponent:
$\alpha \le {1 \over 2} \left(1+{\log b_G \over \log b_F }\right) <1.$

The theory has been also generalized to other combinatorial types in \cite{Haz}, and also to three dimensional dissipative H\'enon-like 
maps in \cite{Ywn}.

Finally, the authors of \cite{dCLM} show that the geometry of these Cantor sets is rather particular: the Cantor sets have universal  bounded geometry in ``most''  places, however there are places in the Cantor set were the geometry is unbounded. Rigidity and universality as we know from one-dimensional dynamics has a probabilistic nature for strongly dissipative H\'enon like maps. See \cite{LM2} for a discussion 
of probabilistic universality and probabilistic rigidity.

It turns out that the period-doubling renormalization for area-preserving maps is very different from the dissipative case.

A universal period-doubling cascade in families of area-preserving maps was observed by several authors in the early 80's \cite{DP,Hel,BCGG,Bou,CEK2,EKW1}.  The existence of a hyperbolic fixed point for the period-doubling renormalization operator 
$$R_{EKW}[F]=\Lambda^{-1}_F \circ F \circ F  \circ \Lambda_F,$$
where $\Lambda_F(x,u)=(\lambda_F x, \mu_F u)$ is an $F$-dependent {\it linear} change of coordinates, has been proved with computer-assistance in \cite{EKW2}. 

We have proved in \cite{GJ2}  that {\it infinitely renormalizable} maps in a neighborhood of the fixed point of \cite{EKW2} admit  a ``stable'' Cantor set, that is the set on which the Lyapunov exponents are zero. We have also shown in the same publication  that the conjugacy of stable dynamics is at least bi-Lipschitz on a submanifold of locally infinitely renormalizable maps  of a finite codimension. Furthermore,  \cite{GJM} improves  this conclusion in the following way.

\bigskip

\noindent
{\bf Rigidity for Area-preserving Maps.} {\it The period doubling Cantor sets of area-preserving maps in  the universality class of the Eckmann-Koch-Wittwer renormalization fixed point are smoothly conjugate.}

\bigskip

A crucial ingredient of the proof in \cite{GJM} is a new tight bound on the spectral radius of the renormalization operator. The goal of the present paper is to prove this new bound.

 We demonstrate that the spectral radius of the action  of $D R_{EKW}$, evaluated at the Eckmann-Koch-Wittwer fixed point $F_{EKW}$, restricted to the tangent space $T_{F_{EKW}}\cW$ of the stable manifold $\cW$ of the infinitely renormalizable maps, is equal exactly to the absolute value of the `` horizontal'' scaling parameter 
$$\rho_{\rm spec} \left(D R_{EKW}[F_{EKW}] \arrowvert_{T_{F_{EKW}}\cW}\right)=|\lambda_{F_{EKW}}|=0.2488 \ldots.$$

Furthermore, we show that the single  eigenvalue $\lambda_{F_{EKW}}$ in the spectrum of $D R_{EKW}[F_{EKW}]$ corresponds to an eigenvector, generated by a very specific  coordinate change.  
To eliminate this  {\it irrelevant}  eigenvalue from the renormalization spectrum, we introduce an $F$-dependent {\it nonlinear} coordinate change $S_F$ into the  period-doubling renormalization scheme
$$R_c[F]:=\Lambda^{-1}_F \circ S^{-1}_F \circ F \circ F  \circ S_F \circ \Lambda_F,$$
compute the spectral radius of the restriction of the spectrum of $D R_c [F^*]$ to its stable subspace $T_{F^*} \cW$ at the fixed point $F^*$ of $R_c$,  and obtain the following  spectral bound, which is of crucial importance to our proof of rigidity.

\bigskip

\noindent {\bf Main Theorem.}
$$\rho_{\rm spec}\left(D R_c [F^*] \arrowvert_{T_{F^*} \cW}\right) \le 0.1258544921875.$$

\section*{Acknowledgment}
This work was started during a visit by the authors to the Institut Mittag-Lefler (Djursholm, Sweden) as part of the research program on ``Dynamics and PDEs''. The hospitality of the institute is gratefully acknowledged. The second author was funded by a postdoctoral fellowship from the Institut Mittag-Lefler, he is currently funded by a postdoctoral fellowship from \textit{Vetenskapsr\aa det} (the Swedish Research Council).

\setcounter{section}{0}

\bigskip

\section{Renormalization for area-preserving reversible twist maps} 

An ``area-preserving map'' will mean an exact symplectic diffeomorphism of a subset of ${\fR}^2$ onto its image.

Recall, that an area-preserving map that satisfies the twist condition
$$\partial_u \left( \pi_x F(x,u) \right) \ne 0$$
everywhere in its domain of definition can be uniquely specified by a generating function $S$:
\begin{equation}\label{gen_func}
\left( x \atop -S_1(x,y) \right) {{ \mbox{{\small \it  F}} \atop \mapsto} \atop \phantom{\mbox{\tiny .}}}  \left( y \atop S_2(x,y) \right), \quad S_i \equiv \partial_i S.
\end{equation}

Furthermore, we will assume that $F$ is reversible, that is 
\begin{equation}\label{reversible}
T \circ F \circ T=F^{-1}, \quad {\rm where} \quad T(x,u)=(x,-u).
\end{equation}

For such maps it follows from $(\ref{gen_func})$ that 
$$S_1(y,x)=S_2(x,y) \equiv s(x,y),$$
and
\begin{equation}\label{sdef}
\left({x  \atop  -s(y,x)} \right)  {{ \mbox{{\small \it  F}} \atop \mapsto} \atop \phantom{\mbox{\tiny .}}} \left({y \atop s(x,y) }\right).
\end{equation}

It is this ``little'' $s$ that will be referred to below as ``the generating function''. If the equation $-s(y,x)=u$ has a unique differentiable solution $y=y(x,u)$, then the derivative of such a map $F$ is given by the following formula:

\begin{equation}\label{Fder}
DF(x,u)=\left[ 
\begin{array}{c c}
-{s_2(y(x,u),x) \over s_1(y(x,u),x)} &  -{1 \over s_1(y(x,u),x)} \\
s_1(x,y(x,u))-s_2(x,y(x,u)) {s_2(y(x,u),x) \over s_1(y(x,u),x)}  & -{s_2(x,y(x,u)) \over s_1(y(x,u),x)} 
\end{array}
\right]. 
\end{equation}

The period-doubling phenomenon can be illustrated with the area-preserving H\'enon family (cf. \cite{Bou}) :
$$ H_a(x,u)=(-u +1 - a x^2, x).$$

Maps $H_a$ have a fixed point $((-1+\sqrt{1+a})/a,(-1+\sqrt{1+a})/a) $ which is stable (elliptic) for $-1 < a < 3$. When $a_1=3$ this fixed point becomes hyperbolic: the eigenvalues  of the linearization of  the map at the fixed point bifurcate through $-1$ and become real.  At the same time a stable orbit of period two is ``born'' with $H_a(x_\pm,x_\mp)=(x_\mp,x_\pm)$, $x_\pm= (1\pm \sqrt{a-3})/a$. This orbit, in turn, becomes hyperbolic at $a_2=4$, giving birth to a period $4$ stable orbit. Generally, there  exists a sequence of parameter values $a_k$, at which the orbit of period $2^{k-1}$ turns unstable, while at the same time a stable orbit of period $2^k$ is born. The parameter values $a_k$ accumulate on some $a_\infty$. The crucial observation is that the accumulation rate
\begin{equation}
\lim_{k \rightarrow \infty}{a_k-a_{k-1} \over  a_{k+1}-a_k } = 8.721...
\end{equation} 
is universal for a large class of families, not necessarily H\'enon.

Furthermore, the $2^k$ periodic orbits scale asymptotically with two scaling parameters
\begin{equation}
\lambda=-0.249 \ldots,\quad \mu=0.061 \ldots
\end{equation}

To explain how orbits scale with $\lambda$ and $\mu$ we will follow \cite{Bou}. Consider an interval $(a_k,a_{k+1})$ of parameter values in a ``typical'' family $F_a$. For any value $\alpha \in (a_k,a_{k+1})$ the map $F_\alpha$ possesses a stable periodic orbit of period $2^{k}$. We fix some $\alpha_k$ within the interval $(a_k,a_{k+1})$ in some consistent way; for instance, by requiring that  $DF^{2^{k}}_{\alpha_k}$ at a point in the stable $2^{k}$-periodic orbit is conjugate, via a diffeomorphism $H_k$, to a rotation with some fixed rotation number $r$.  Let $p'_k$ be some unstable periodic point in the $2^{k-1}$-periodic orbit, and let $p_k$ be the further of the two stable $2^{k}$-periodic points that bifurcated from $p'_k$.  Denote with $d_k=|p'_k-p_k|$, the distance between $p_k$ and $p'_k$. The new elliptic point $p_k$ is surrounded by (infinitesimal) invariant ellipses; let $c_k$ be the distance between $p_k$ and $p'_k$ in t
 he direction of  the minor semi-axis of an invariant ellipse surrounding $p_k$, see Figure \ref{bifGeom}. Then,
$${1 \over \lambda}=-\lim_{k \rightarrow \infty}{ d_k \over d_{k+1}},\quad
{\lambda \over \mu}=-\lim_{k \rightarrow \infty}{ \rho_k \over \rho_{k+1}}, \quad
{1 \over \lambda^2}=\lim_{k \rightarrow \infty}{ c_k \over c_{k+1}},
$$
where $\rho_k$ is the ratio of the smaller and larger eigenvalues of $D H_k(p_k)$. 

\begin{figure}[t]
\psfrag{p_k}{$p_k$}
\psfrag{p'_k}{$p'_k$}
\psfrag{c_k}{$c_k$}
\psfrag{d_k}{$d_k$}
\begin{center}
\includegraphics[width=0.7 \textwidth]{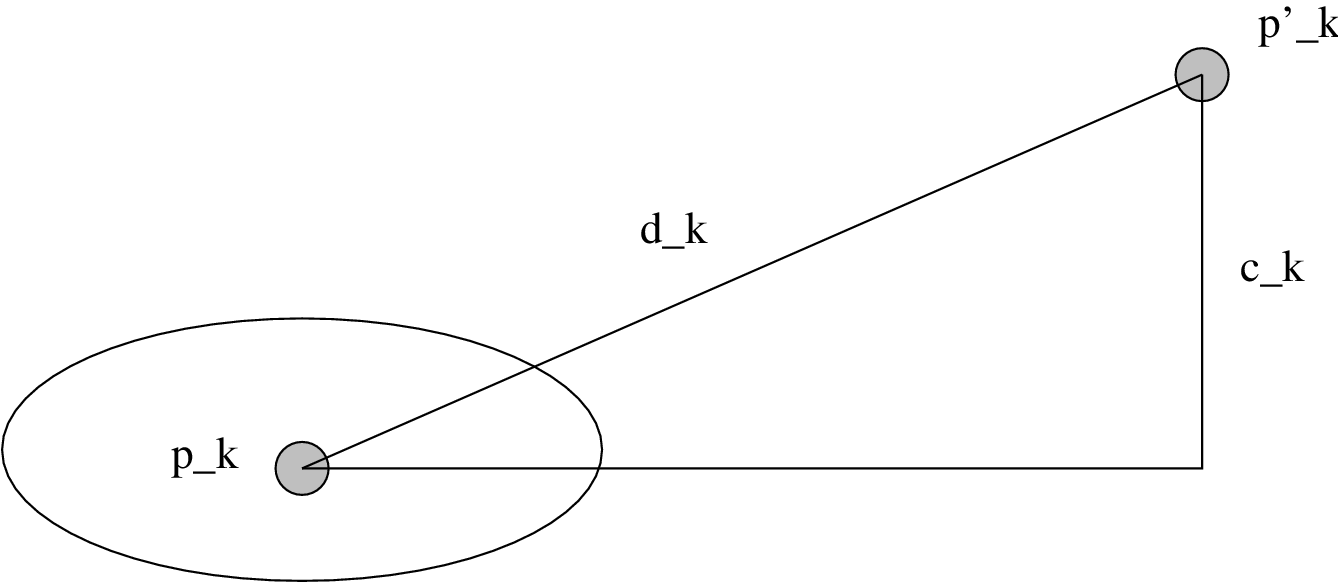}
\caption{The geometry of the period doubling. $p_k$ is the further elliptic point that has bifurcated from the hyperbolic point $p'_k$.}\label{bifGeom}
\end{center}
 \end{figure}

This universality can be explained rigorously if one shows that the {\it renormalization} operator
\begin{equation}\label{Ren}
R_{EKW}[F]=\Lambda^{-1}_F \circ F \circ F \circ \Lambda_F,
\end{equation}
where $\Lambda_F$ is some $F$-dependent coordinate transformation, has a fixed point, and the derivative of this operator is hyperbolic at this fixed point.

It has been argued in \cite{CEK2}  that $\Lambda_F$ is a diagonal linear transformation. Furthermore, such $\Lambda_F$ has been used in \cite{EKW1} and \cite{EKW2} in a computer assisted proof of existence of a reversible renormalization fixed point $F_{EKW}$ and hyperbolicity of the operator $R_{EKW}$.

We will now derive an equation for the generating function of the renormalized map $\Lambda_F^{-1} \circ F \circ F \circ \Lambda_F$.

Applying a reversible $F$ twice we get
$$
 \left({x'  \atop  -s(Z,x')} \right) {{ \mbox{{\small \it  F}} \atop \mapsto} \atop \phantom{\mbox{\tiny .}}} \left({Z \atop s(x',Z)} \right) = \left({Z  \atop  -s(y',Z)} \right) {{ \mbox{{\small \it  F}} \atop \mapsto} \atop \phantom{\mbox{\tiny .}}} \left({y'\atop  s(Z,y')} \right).
$$

According to \cite{CEK2}  $\Lambda_F$ can be chosen to be a linear diagonal transformation:  

$$\Lambda_F(x,u)=(\lambda x, \mu u).$$

We, therefore, set  $(x',y')=(\lambda x,  \lambda y)$, $Z(\lambda x, \lambda y)= z(x,y)$ to obtain:

\begin{equation}\label{doubling}
\left(\!{x  \atop  -{ 1 \over \mu } s(z,\lambda x)} \!\right) \!{{ \mbox{{\small $\Lambda_F$}} \atop \mapsto} \atop \phantom{\mbox{\tiny .}}} \!\left(\!{\lambda x  \atop  -s(z,\lambda x)} \!\right) \!{{ \mbox{{\small \it  F $ \circ$ F}} \atop \mapsto} \atop \phantom{\mbox{\tiny .}}}\!\left(\!{\lambda y \atop s(z,\lambda y)}\! \right)   {{ \mbox{{\small \it  $\Lambda_F^{-1}$}} \atop \mapsto} \atop \phantom{\mbox{\tiny .}}} \left(\!{y \atop {1 \over \mu } s(z,\lambda y) }\!\right),
\end{equation}
where $z(x,y)$ solves
\begin{equation}\label{midpoint}
s(\lambda x, z(x,y))+s(\lambda y, z(x,y))=0.
\end{equation}

If the solution of $(\ref{midpoint})$ is unique, then $z(x,y)=z(y,x)$, and it follows from $(\ref{doubling})$ that the generating function of the renormalized $F$ is given by 
\begin{equation}
\tilde{s}(x,y)=\mu^{-1} s(z(x,y),\lambda y).
\end{equation}

One can fix a set of normalization conditions for $\tilde{s}$ and $z$ which serve to determine scalings $\lambda$ and $\mu$ as functions of $s$. For example, the normalization $s(1,0)=0$ is reproduced for $\tilde{s}$ as long as $z(1,0)=z(0,1)=1.$ In particular, this implies that 
$$s(Z(\lambda,0),0)=0,$$
which serves as an equation for $\lambda$.  Furthermore, the condition $\partial_1 s(1,0)=1$ is reproduced as long as $\mu=\partial_1 z (1,0).$

We will now summarize the above discussion in the following definition of the renormalization operator acting on generating functions originally due  to the authors of \cite{EKW1} and \cite{EKW2}:

\begin{definition}\label{EKW_def}
Define the prerenormalization of $s$ as
\begin{eqnarray}\label{preren_eq}
\nonumber \\ {\cP}_{EKW}[s]=s \circ G[s],
\end{eqnarray}
where
\begin{eqnarray}
\label{midpoint_eq} 0&=&s(x, Z(x,y))+s(y, Z(x,y)),\\
\label{G_def} G[s](x,y)&=&(Z(x,y),y).
\end{eqnarray}

The renormalization of $s$ will be defined as
\begin{eqnarray}\label{ren_eq}
\nonumber \\ {\cR}_{EKW}[s]={1 \over \mu} {\cP}_{EKW}[s] \circ \lambda,
\end{eqnarray}
where
\begin{equation}
\nonumber \lambda(x,y)=(\lambda x, \lambda y), \quad \cP_{EKW}[s](\lambda,0)=0 \quad {\rm and} \quad \mu=\lambda \ \partial_1 \cP_{EKW}[s](\lambda,0).
\end{equation}

\end{definition}

\medskip

\begin{definition}\label{B_space}
The Banach space of functions  $s(x,y)=\sum_{i,j=0}^{\infty}c_{i j} (x-\beta)^i (y-\beta)^j$, analytic on a bi-disk
$$\cD_\rho(\beta)=\{(x,y) \in \field{C}^2: |x-\beta|<\rho, |y-\beta|<\rho\},$$
for which the norm
$$\|s\|_\rho=\sum_{i,j=0}^{\infty}|c_{i j}|\rho^{i+j}$$
is finite, will be referred to as $\cA^\beta(\rho)$.

$\cA_s^\beta(\rho)$ will denote its symmetric subspace $\{s\in\cA^\beta(\rho) : s_1(x,y)=s_1(y, x)\}$.

We will use the simplified notation $\cA(\rho)$  and  $\cA_s(\rho)$ for  $\cA^0(\rho)$ and $\cA_s^0(\rho)$, respectively. 
\end{definition}

\medskip

As we have already mentioned, the following has been proved with the help of a computer in \cite{EKW1} and \cite{EKW2}:
\begin{oldthm}\label{EKWTheorem}
There exist a polynomial $s_{0.5} \in \cA_s^{0.5}(\rho)$ and  a ball $\cB_\varrho(s_{0.5}) \subset \cA_s^{0.5}(\rho)$, $\varrho=6.0 \times 10^{-7}$, $\rho=1.6$, such that the operator ${\cR}_{EKW}$ is well-defined and analytic on $\cB_\varrho(s_{0.5})$. 

Furthermore, its derivative $D {\cR}_{EKW} \arrowvert_{\cB_\varrho(s_{0.5})}$ is a compact linear operator, and has exactly two eigenvalues 
$$\delta_1=8.721..., \quad {\rm and}$$
$$\delta_2={1\over \lambda_*}$$
\noindent of modulus larger than $1$, while 
$${\rm spec}(D {\cR}_{EKW} \arrowvert_{\cB_\varrho(s_{0.5})}) \setminus \{\delta_1,\delta_2 \} \subset \{z \in \C: |z| \le \nu \},$$ 
where 
\begin{equation}\label{contr_rate}
\hspace{3.0cm}\nu <  0.85.
\end{equation}


Finally, there is an $s^{EKW} \in \cB_\varrho(s_{0.5})$ such that
$$\cR_{EKW}[s^{EKW}]=s^{EKW}.$$
The scalings $\lambda_*$ and $\mu_*$ corresponding to the fixed point $s^{EKW}$ satisfy
\begin{eqnarray}
\label{lambda} \lambda_* \in [-0.24887681,-0.24887376], \\
\label{mu} \mu_* \in [0.061107811, 0.061112465].
\end{eqnarray}
 \end{oldthm}

\medskip
 \begin{remark}
The bound $(\ref{contr_rate})$ is not sharp. In fact, a bound on the largest eigenvalue of $D {\cR}_{EKW}(s^{EKW})$, restricted to the tangent space of the stable manifold, is expected to be quite smaller.
 \end{remark}
\medskip

The size of the neighborhood in $\cA_s^\beta(\rho)$ where the operator $\cR_{EKW}$ is well-defined, analytic and compact has been improved in \cite{Gai4}. Here, we will cite a somewhat different version of the result of \cite{Gai4} which suits the present discussion  (in particular, in the Theorem below some parameter, like $\rho$ in $\cA_s^\beta(\rho)$, are different from those used in \cite{Gai4}). We would like to emphasize that all parameters and bounds used and reported in the Theorem below, and, indeed, throughout the paper, are numbers representable on the computer.

\begin{oldthm}\label{MyTheorem}$\phantom{aa}$\\
There exists a polynomial $s^0  \in \cA(\rho)$, $\rho=1.75$, such that the following holds.

\noindent {\it i)}  The operator $\cR_{EKW}$ is well-defined and analytic in $\cB_R(s^0) \subset \cA(\rho)$ with 
$$R=0.00426483154296875.$$

\medskip

\noindent {\it ii)} For all $s \in \cB_R(s^0)$ with real Taylor coefficients, the scalings  $\lambda=\lambda[s]$ and $\mu=\mu[s]$ satisfy
\begin{eqnarray}
\nonumber  0.0000253506004810333 \le & \mu & \le 0.121036529541016,\\
 \nonumber  -0.27569580078125 \le  & \lambda & \le -0.222587585449219.
\end{eqnarray}  

\medskip

\noindent {\it iii)} The operator $\cR_{EKW}$ is compact in $\cB_R(s^0) \subset \cA(\rho)$, with $\cR_{EKW}[s] \in \cA(\rho')$, $\rho'=1.0699996948242188 \rho$.
\end{oldthm}

\medskip

\begin{definition}\label{Fdefr}
The set of reversible twist maps $F$ of the form (\ref{sdef})  with $s \in \cB_\varrho(\tilde{s}) \subset \cA_s^\beta(\rho)$  will be referred to as $\cF^{\beta,\rho}_\varrho(\tilde{s})$:
\begin{equation}\label{F-set}
\cF^{\beta,\rho}_\varrho(\tilde{s})=\left\{ F: (x,-s(y,x)) \mapsto (y,s(x,y))| \quad s \in \cB_\varrho(\tilde{s}) \subset \cA_s^{\beta}(\rho)\right\}.
\end{equation}

We will also use the notation 
$$\cF^{\rho}_\varrho(\tilde{s}) \equiv \cF^{0,\rho}_\varrho(\tilde{s}).$$
\end{definition}

\bigskip

We will finish our introduction into period-doubling for area-preserving maps  with  a summary of properties of the fixed point map. In \cite{GJ1} we have described the domain of analyticity of maps in some neighborhood of the fixed point. Additional properties of the domain are studied in \cite{J1}. Before we state the results of \cite{GJ1}, we will fix a notation for spaces of functions analytic on a subset of $\field{C}^2$.

\begin{definition}\label{tr_norms}
Denote $\cO_2(\cD)$ the Banach space of maps $F: \cD \mapsto \field{C}^2$, analytic on an open simply connected set $\cD \subset \field{C}^2$, continuous on $\partial \cD$, equipped with a finite max supremum norm $\| \cdot \|_\cD$:
$$\| F \|_{\cD}=\max\left\{\sup_{(x,u) \in \cD}|F_1(x,u)|, \sup_{(x,u) \in \cD}|F_2(x,u)| \right\}.$$

The Banach space of functions $y: \cA \mapsto \field{C}$, analytic on an open simply connected set $\cA \subset \field{C}^2$, continuous on $\partial \cA$, equipped with a finite supremum norm $\| \cdot \|_\cA$ will be denoted  $\cO_1(\cA)$:
$$\| y \|_{\cD}=\sup_{(x,u) \in \cD} |y(x,u)|.$$

If $\cD$ is a bidisk $\cD_{\rho} \subset \field{C}^2$ for some $\rho$, then we use the notation
$$\| \cdot \|_{\rho}\equiv\| \cdot  \|_{\cD_{\rho}}.$$

\end{definition}

The next Theorem describes the analyticity domains for maps in a neighborhood of the Eckmann-Koch-Wittwer fixed point map, and those for functions in a neighborhood of the Eckmann-Koch-Wittwer fixed point generating function. The Theorem has been proved in two different versions: one for the space $\cA^{0.5}_s(1.6)$ (the functional space in the original paper \cite{EKW2}), the other for the space $\cA_s(1.75)$ --- the space in which we will obtain a bound on  the renormalization spectral radius in the stable manifold in this paper. To state the Theorem  in a compact form, we introduce the following notation:
$$\rho_{0.5}=1.6, \quad \rho_0=1.75,$$
$$\varrho_{0.5}=6.0 \times 10^{-7}, \quad \varrho_0=5.79833984375 \times 10^{-4},$$
while $s_{0.5}$ (as in Theorem $\ref{EKWTheorem}$) and $s_0$ will denoted the approximate renormalization fixed points in spaces $\cA_s^{0.5}(1.6)$ and $\cA_s(1.75)$,  respectively.

\begin{oldthm}\label{fp_properties}
There exists a polynomial  $s_\beta$ such that  the following holds for all  $F \in  \cF^{\beta,\rho_\beta}_{\varrho_\beta}(s_\beta)$, $\beta=0.5$ or $\beta=0$.
\medskip

\noindent i)  There exists a simply connected open set $\cD=\cD(\beta,\varrho_\beta,\rho_\beta) \subset \fC^2$ such that the map $F$ is in $\cO_2(\cD)$.  

\medskip

\noindent ii)  There exist simply connected open sets $\bar{\cD}=\bar{\cD}(\beta,\varrho_\beta,\rho_\beta)  \subset \cD$, such that $\bar{\cD} \cap \R^2$ is a non-empty simply connected open set, and  such that for every $(x,u) \in \bar{\cD}$ and $s \in \cB_{\varrho_\beta}(s_\beta) \subset \cA_s^\beta(\rho_\beta)$, the equation 
\begin{equation}\label{y_equation}
0=u+s(y,x)
\end{equation}
has a unique solution $y[s](x,u) \in \cO_1(\bar{\cD})$. The map 
$$\cS: s \mapsto y[s]$$ 
is analytic as a map from $ \cB_{\varrho_\beta}(s_\beta)$ to $\cO_1(\bar{\cD})$.

Furthermore, for every $x \in \pi_x \bar{\cD}$, there is a function $U \in \cO_1(\cD_{\rho_\beta}(\beta))$, that satisfies
$$y[s](x,U(x,v))=v.$$
The map 
$$Y: y[s] \mapsto U$$
is analytic as a map from $\cO_1(\cD_{\rho_\beta}(\beta))$ to $\cB_{\varrho_\beta}(s_\beta)$.
\end{oldthm}

\begin{remark}\label{mapI}
It is not too hard to see that the subsets $\cF^{\beta,\rho_\beta}_{\varrho_\beta}(s_\beta)$, $\beta=0$ or $0.5$, are  analytic Banach submanifolds of  the spaces $\cO_2(\cD(\beta,\varrho_\beta,\rho_\beta)$. Indeed, the map
\begin{eqnarray}
\cI: s \mapsto \left(y[s],s \circ h[s]\right),
\end{eqnarray}
where $y[s](x,u)$ is the solution of the equation $(\ref{y_equation})$,  and $h[s](x,u)=(x,y[s](x,u))$,  is analytic as a map from $\cB_{\varrho_\beta}(s_\beta)$ to $\cO_2(\cD(\beta,\varrho_{\beta},\rho_\beta)$ according to Theorem $\ref{fp_properties}$, and has an analytic inverse
\begin{eqnarray}
\cI^{-1}: F \mapsto \pi_u F \circ g[F],
\end{eqnarray}
where $g[F](x,y)=(x,U(x,y))$, and $U$ is as in Theorem $\ref{fp_properties}$. 

\end{remark}

\medskip

We are now ready to give a definition of the Eckmann-Koch-Wittwer renormalization operator for maps of the subset of a plane. Notice, that the condition $\cP_{EKW}[s](\lambda,0)=0$ from Definition $\ref{EKW_def}$ is equivalent to 
$$F(F(\lambda,-s(z(\lambda,0),\lambda)))=(0,0),$$
or, using the reversibility
$$\lambda=\pi_x F(F(0,0)).$$
On the other hand, 
$$-s(z(y(x,u),x),x)=-\cP_{EKW}[s](y(x,u),x)=u,$$
and
\begin{eqnarray}
\nonumber \partial_u \cP_{EKW}[s](y(x,u),x)&=&\cP_{EKW}[s]_1(y(x,u),x) y_2(x,u)\\
\nonumber &=&\cP_{EKW}[s]_1(y(x,u),x) \ \pi_x (F \circ F)_2(x,u) =-1,
\end{eqnarray}
then
$$\cP_{EKW}[s]_1(\lambda,0) \ \pi_x (F \circ F)_2(0,0) =-1,$$
and
$$\mu={-\lambda \over  \pi_x (F \circ F)_2(0,0)}.$$

\medskip

\begin{definition}
We will refer to the composition $F \circ F$ as the {\it prerenormalization} of $F$, whenever this composition is defined:
\begin{equation}
P_{EKW}[F]=F\circ F.
\end{equation}

Set
\begin{equation} \label{R_EKW_2}
\nonumber R_{EKW}[F]=\Lambda^{-1} \circ P_{EKW}[F] \circ  \Lambda, 
\end{equation}
where
$$\Lambda(x,u)=(\lambda x, \mu u), \quad \lambda=\pi_x P_{EKW}[F](0,0), \quad \mu={-\lambda \over  \pi_x P_{EKW}[F]_2(0,0)},$$
whenever  these operations are defined. $R_{EKW}[F]$ will be called the (EKW-)renormalization of $F$.
\end{definition}

\begin{remark}\label{remarkI}
Suppose that for some choice of $\beta$, $\varrho_\beta$ and $\rho_\beta$,  the operator $\cR_{EKW}$ and the map $\cI$, described  in Remark $\ref{mapI}$, are well-defined on some $\cB_{\varrho_\beta}(s_\beta) \subset \cA^\beta_s(\rho_\beta)$.  Also, suppose that the inverse of $\cI$ exists on $\cI(\cB_{\varrho_\beta}(s_\beta))$. Then, 
$$R_{EKW} = \cI \circ    \cR_{EKW}   \circ   \cI^{-1}$$
on $\bF$.
\end{remark}

\bigskip \section{Statement  of main results} \label{Results}

Consider the coordinate transformation 
\begin{equation}\label{coords}
 \nonumber S_t(x,u)=\left(x+t x^2,{u \over 1+2 t x}\right), \quad  S_t^{-1}(y,v)=\left({\sqrt{1+4t y}-1 \over 2 t}, v \sqrt{1+4 t y}  \right),
\end{equation}
for $t \in \C$, $|t|<4/(\rho+|\beta|)$ (recall Definition $\ref{B_space}$).

We will now introduce two renormalization operators, one -  on the generating functions, and one - on the maps, which incorporates the coordinate change $S_t$ as an additional coordinate transformation.

\begin{definition}\label{simple_strong_renorm}
Given $c \in \R$, set, formally,
$$\nonumber  \cP_c[s](x,y)=( 1+ 2 t_c y) s(G(\xi_{t_c}(x,y))), \quad {\rm and} \quad \cR_c[s]=\mu^{-1} \cP_c[s] \circ \lambda,$$
with $G$ is as in $(\ref{G_def})$, and
\begin{equation}
\xi_t(x,y)=(x+t x^2,y+t y^2), \quad  t_c[s]={1 \over 4} {c-\left(s \circ G\right)_{(0,3)} \over \left(s \circ G\right)_{(0,2)}},
\end{equation}
where  $\la$ and $\mu$ solve the following equations:
\begin{equation}
\cP_c[s](\lambda[s],0)=0,\quad  \mu[s] = \lambda[s] \partial_1 \cP_c[s](\lambda[s],0).
\end{equation}
\end{definition}

\begin{definition}\label{simple_strong_Frenorm}
Given $c \in \R$, set, formally,
\begin{equation}
\label{simple_Fren_eq}  P_c [F]= S_{t_c}^{-1} \circ F \circ F \circ S_{t_c}, \quad R_c[F]=\Lambda^{-1}_F \circ P_c[F] \circ \Lambda_F,
\end{equation}
where $S_{t_c}$  is as in $(\ref{coords})$, $\Lambda_F(x,u)=(\lambda[F] x, \mu[F] u)$, and 
$$t_c[F]\!=\!{1 \over 4} {c-\left(\pi_u (F  \circ F)\right)_{(0,3)} \over \left(\pi_u (F  \circ F)\right)_{(0,2)} },  \quad  \lambda[F]\!=\!\pi_x P_c[F](0,0),  \quad  \mu[F]\!=\!{-\lambda[F] \over  \pi_x P_c[F]_2(0,0)}.$$
\end{definition}

We are now ready to state our main theorem.  Below, and through the paper, $s_{(i,j)}$ stands for the $(i,j)$-th component of a Taylor series expansion of an analytic function of two variables. 

\bigskip 

\begin{mainthm}\label{Main_Theorem} (Existence and Spectral properties)
There exists a polynomial $s_0: \field{C}^2 \mapsto \field{C}$, such that 
\bigskip 
\begin{itemize}
\item[$i)$] The operators $\cR_{EKW}$ and $\cR_{c_0}$, where $c_0= \left(s_0 \circ G[s_0] \right)_{(0,3)}$,  are well-defined, analytic and compact in $\cB_{\varrho_0}(s_0) \subset \cA_s(\rho)$, with
$$\rho=1.75, \quad \varrho_0=5.79833984375 \times 10^{-4}.$$

\medskip 

\item[$ii)$] There exists a function $s^* \in \cB_r(s_0) \subset \cA_s(\rho)$ with
$$r=1.1 \times 10^{-10},$$
such that
$$\cR_{c_0}[s^*]=s^*.$$ 

\medskip 

\item[$iii)$] The linear operator  $D \cR_{c_0}[s^{*}]$  has two eigenvalues outside of the unit circle:
$$ 8.72021484375 \le \delta_1 \le 8.72216796875, \quad  \delta_2={1 \over \lambda_*},$$
where 
$$  -0.248875313689    \le \lambda_* \le  -0.248886108398438.$$

\medskip 

\item[$iv)$] The complement of these two eigenvalues in the spectrum is compactly contained in the unit disk:
$$
{\rm spec}(D \cR_{c_0}[s^*]) \setminus \{\delta_1,\delta_2\} \subset \{z \in \C: |z| \le 0.1258544921875 \equiv \nu\}.
$$

\end{itemize}
\end{mainthm}

\bigskip

The Main Theorem implies that there exist  codimension $2$ local stable manifolds $\cW_{\cR_{c_0}}(s^*) \subset \cA_s(1.75)$,  
such that the contraction rate in $\cW_{\cR_{c_0}}(s^*)$ is bounded from above by $\nu$:
$$\|\cR_{c_0}^n[s]-\cR_{c_0}^n[\tilde{s}]\|_\rho=O(\nu^n)$$
for any two $s$ and $\tilde{s}$ in $\cW_{\cR_{c_0}}(s^*)$.

\begin{definition}\label{Wdeff} $\phantom{aa}$\\

\noindent $i)$  The set of reversible twist maps of the form (\ref{sdef}) such that $s \in \cW_{\cR_{c_0}}(s^*) \subset \cA_s(1.75)$ will be denoted $W$, and referred to as  \underline{infinitely renormalizable maps}.

\bigskip

\noindent $ii)$ Set, $W_\varrho(s_0) \equiv  W \cap \cF^{1.75}_\varrho(s_0)$, where $\cF^{1.75}_\varrho(s_0)$ is as in  Definition $\ref{Fdefr}$.
\end{definition}

Naturally, these sets are  invariant under renormalization if $\varrho$ is sufficiently small.

\bigskip

Notice, that, among other things, this Theorem restates the result about existence of the Eckmann-Koch-Wittwer fixed point and renormalization hyperbolicity of Theorem $\ref{EKWTheorem}$ in a setting of a different functional space. We do not prove that the fixed point $s^*$, after an small adjustment corresponding to the coordinate change $S_t$,  coincides with $s^{EKW}$ from Theorem $\ref{EKWTheorem}$, although the computer bounds on these two  fixed points differ by a tiny  amount  on any bi-disk contained in the intersection of their domains.

The fact that the operator $R_{c_0}$ as in $(\ref{simple_Fren_eq})$ contains an additional coordinate change does not cause a problem: conceptually, period-doubling renormalization of a map is its  second iterate conjugated by a coordinate change,  which does not have to be necessarily linear.

\bigskip \section{Coordinate changes and renormalization eigenvalues} 

Let $\cD$ and $\bar{\cD}$ be as in the Theorem $\ref{fp_properties}$.  Consider the action of the operator 
\begin{equation} \label{R_star}
R_*[F]=\Lambda^{-1}_* \circ F \circ F \circ  \Lambda_*
\end{equation}
on $\cO_2(\cD)$, where
$$\Lambda_*(x,u)=(\lambda_*  x, \mu_* u),$$
with $\lambda_*$ and $\mu_*$ being the fixed scaling parameters corresponding to the Collet-Eckmann-Koch as in Theorem $\ref{EKWTheorem}$.

According to Theorem $\ref{EKWTheorem}$ this operator is analytic and compact on the subset $\cF^{0.5,1.6}_\varrho(s_{0.5})$, $\varrho=6.0 \times 10^{-7}$, of $\cO_2(\cD)$, and has a fixed point $F_{EKW}$.  In this paper, we will prove the existence of a fixed point $s^*$ of the operator $\cR_{EKW}$ in a Banach space different from that in  Theorem $\ref{EKWTheorem}$. Therefore, we will state most of our results concerning the spectra of renormalization operators for general spaces  $\cA^{\beta}_s(\rho)$ and sets $\aF$,  under the hypotheses of existence of a fixed point $s^*$, and analyticity and compactness of the operators in some neighborhood of the fixed point. Later, a specific choice of parameters $\beta$, $\rho$ and $\varrho$ will  be made,  and the hypotheses - verified.

Let $S=id+\sigma$ be a coordinate transformation of the domain $\cD$ of maps $F$, satisfying
$$D S  \circ F=D S.$$

In particular, these transformations preserve the subset of area-preserving maps.

Notice, that
\begin{eqnarray}
\nonumber (id +\epsilon \sigma)^{-1} \circ F \circ (id +\epsilon \sigma)
&=& F + \epsilon  \left( -\sigma \circ F + D F \cdot \sigma\right)+O(\epsilon^2) \\
\nonumber &\equiv& F+\epsilon h_{F,\sigma}+O(\epsilon^2).
\end{eqnarray}

Suppose that the operator $R_*$ has a fixed point $F^*$  in some neighborhood $\cB \subset \cO_2(\cD)$, on which $R_*$ is analytic and compact. 
Consider the action $D R_*[F] h_{F,\sigma}$ of the derivative of this operator. 
\begin{eqnarray}
\nonumber D R_*[F] h_{F,\sigma}&=& \partial_\epsilon \left(\Lambda_*^{-1} \circ (F+\epsilon h_\sigma)  \circ (F+\epsilon h_\sigma) \circ \Lambda_* \right) \arrowvert_{\epsilon=0}\\
\nonumber &=&\partial_\epsilon \left(\Lambda_*^{-1} \circ (id +\epsilon \sigma)^{-1} \circ F \circ F \circ (id +\epsilon \sigma \circ \Lambda_* \right) \arrowvert_{\epsilon=0} \\
\nonumber &=& \Lambda_*^{-1} \cdot \left[-\sigma  \circ F \circ F + D(F\circ F) \cdot \sigma \right] \circ \Lambda_* \\
\label{derivation} &=& \Lambda_*^{-1} \cdot h_{F \circ F,\sigma} \circ \Lambda_*.
\end{eqnarray}

Specifically, if $F=F^*$, one gets
$$ D R_*[F^*] h_{F^*,\sigma}=h_{F^*,  \tau}, \quad \tau=\Lambda_*^{-1} \cdot \sigma \circ \Lambda_*,$$
and clearly, $h_{F^*,\sigma}$ is an eigenvector, if $\tau=\kappa \sigma$, of eigenvalue $\kappa$. In particular, 
$$\kappa=\lambda_*^i \mu_*^j, \quad i \ge 0, \ j \ge 0$$ 
is an eigenvalue of multiplicity (at least) $2$ with eigenvectors $h_{F^*,\sigma}$ generated by
\begin{equation}\label{sigmas}
\sigma^1_{i,j}(x,u)=(x^{i+1} u^j,0), \quad  \sigma^2_{i,j}(x,u)=(0,x^i u^{j+1}),
\end{equation}
while
$$\kappa=\mu_*^j \lambda_*^{-1}, j \ge 0, \quad  {\rm and} \quad \kappa=\lambda_*^i \mu^{-1}_*, i \ge 0,$$ 
are each eigenvalues of multiplicity (at least) $1$, generated by 
\begin{equation}\label{sigmas_2}
\sigma^1_{-1,j}(x,u)=(u^j,0), \quad  {\rm and} \quad \sigma^2_{i,-1}(x,u)=(0,x^i),
\end{equation}
respectively.

Next, suppose $S_t^\sigma$, $S_0^\sigma=Id$, is a transformation of coordinates generated by a function $\sigma$ as in $(\ref{sigmas})$-$(\ref{sigmas_2})$, associated with an eigenvalue $\kappa$ of $D R_*[F^*]$. In addition to the operator $(\ref{R_star})$, consider
\begin{equation}\label{R_sigma}
R_\sigma[F]=\Lambda^{-1}_* \circ \left(S^\sigma_{t_\sigma[F]} \right)^{-1} \circ F \circ F \circ S_{t_\sigma[F]}^\sigma \circ  \Lambda_*.
\end{equation}
where the parameter $t_\sigma[F]$ is chosen as
\begin{equation}\label{t_def}
t_\sigma[F]=-{1 \over \kappa \|h_{F^*,\sigma}\|_\cD} \|E(\kappa)(R_*[F]-F^*)\|_{\cD},
\end{equation}
$E(\kappa)$ being the Riesz spectral projection associated with $\kappa$:
$$E(\kappa)={1 \over 2 \pi i} \int_\gamma (z-D R_*[F^*])^{-1} dz$$
($\gamma$  -  a Jordan contour that enclose only $\kappa$ in the spectrum of $D R_*[F^*]$).

We will now compare the spectra of the operators $R_*$ and $R_\sigma$. The result below should be interpreted as follows: if $h_{F^*,\sigma}$ is an eigenvector of $D R_*[F^*]$ generated by a coordinate change $id+\eps \sigma$, and associated with some eigenvalue $\kappa$, then this eigenvalue is eliminated from the spectrum of $D R_\sigma[F^*]$, {\it  if its multiplicity is $1$}.

\begin{lemma}\label{elimination_lemma}
Suppose, there exists a map $F^*$ in some $\cO_2(\cD)$, and a neighborhood $\cB(F^*) \subset \cO_2(\cD)$, such that the operators $R_*$ and $R_\sigma$ are analytic and compact as maps from $\cB(F^*)$ to $\cO_2(\cD)$, and $R_*[F^*]=R_\sigma[F^*]=F^*$.

 Then,
$${\rm spec}(D R_* [F^*])={\rm spec}(D R_\sigma [F^*]) \cup \{\kappa\}.$$

Moreover, if the multiplicity of $\kappa$ is $1$, then
$$
{\rm spec}(D R_* [F^*])\setminus {\rm spec}(D R_\sigma [F^*]) = \{\kappa\}.
$$

\end{lemma}

\begin{proof}

Since $D R_\sigma[F^*]$ and $D R_*[F^*]$ are both compact operators acting on an infinite-dimensional space, their spectra contain $\{0\}$.

Suppose $h$ is a eigenvector of $D R_*[F^*]$ corresponding to some eigenvalue $\delta$, then
\begin{eqnarray}\label{DR_action}
\nonumber D R_\sigma[F^*] h &=&  DR_*[F^*] h
\\
\nonumber &+&  \Lambda_*^{-1} \cdot \left( D_F \left( S_{t_\sigma[F^*]}^\sigma\right)^{-1}  h \right) \circ F^* \circ F^*  \circ  S_{t_\sigma[F^*]}^\sigma \circ \Lambda_* \\
\nonumber  &+&  \Lambda_*^{-1} \cdot \left[ D \left(\left(S_{t_\sigma[F^*]}^\sigma \right)^{-1}  \circ F^* \circ F^* \right)  \circ S_{t_\sigma[F^*]}^\sigma  \cdot \right. \\
\nonumber  &\phantom{+}&  \phantom{\Lambda_*^{-1}}   \cdot \left. \left( D_F S_{t_\sigma[F^*]}^\sigma h\right) \right] \circ \Lambda_*\\
 \nonumber &=& \delta h  +  \Lambda_*^{-1} \cdot \left( D_F \left(S^\sigma_{t_\sigma[F^*]}\right)^{-1} h \right) \circ \Lambda_* \circ  F^* \\
 &\phantom{=}& \phantom{\delta h} + \left[ D F^*  \cdot \Lambda_*^{-1} \cdot \left( D_F S_{t_\sigma[F^*]}^\sigma h \right) \right] \circ \Lambda_*
\end{eqnarray} 
(we have used the fact that $F^*$ satisfies the fixed point equation), where
$$ t_\sigma[F^*]\equiv 0 \quad {\rm and} \quad D_F S_{t_\sigma[F^*]}^\sigma h  \equiv  \partial_\epsilon \left[ S_{t_\sigma[{F^* +\epsilon h}]}^\sigma \right]_{\epsilon=0} =\left(D_F t_\sigma[F^*]h\right) \sigma.$$

More specifically,
\begin{eqnarray}
\nonumber t_\sigma[F^*+\epsilon h]&=& -{\kappa}^{-1} \|h_{F^*,\sigma}\|_\cD^{-1} \|  E(\kappa)\left( R_*(F^*+\epsilon h) -F^*  \right) \|_\cD\\
\nonumber &=&-\epsilon \kappa^{-1}  \|h_{F^*,\sigma}\|_\cD^{-1} \| E(\kappa)\left(D R_* [F^*] h \right)\|_\cD +O(\epsilon^2)\\
\nonumber &=&-\epsilon  \|h_{F^*,\sigma}\|_\cD^{-1} \kappa^{-1} \delta \|  \left( E(\kappa) h \right) \|_\cD +O(\epsilon^2),\\
\nonumber &=&-\epsilon  \|h_{F^*,\sigma}\|_\cD^{-1} \kappa^{-1} \delta \|  \left( E(\kappa) \left( E(\delta) h  \right) \right) \|_\cD +O(\epsilon^2),
\end{eqnarray}
and
\begin{equation}\label{t_der}
D_F t_\sigma[F^*]h=\partial_\epsilon \left[ t_\sigma[F^*+\epsilon h] \right]_{\epsilon=0}=- \|h_{F^*,\sigma}\|_\cD^{-1}  \kappa^{-1} \delta \| \left( E(\kappa) \left( E(\delta) h  \right) \right)  \|_\cD.
\end{equation}

If $\delta=\kappa$ and $h=h_{F^*,\sigma}$ then 
$$D_F t_\sigma[F^*]h=-1$$
(recall, that $E(\delta)^2=E(\delta)$) and
\begin{eqnarray}
\nonumber \Lambda_*^{-1} &\cdot& \left( D_F \left( S^\sigma_{t_\sigma[F^*]} \right)^{-1} h \right) \circ \Lambda_* \circ  F^* + D F^*  \cdot \Lambda_*^{-1} \cdot \left( D_F S_{t_\sigma[F^*]}^\sigma h \right) \circ \Lambda_*\\
\nonumber &=&- \left[ -\Lambda_*^{-1} \cdot \sigma \circ \Lambda_* \circ  F^* + D F^*  \cdot \Lambda_*^{-1} \cdot \sigma \circ \Lambda_*\right]\\
\nonumber &=&- \kappa \left[ - \sigma \circ  F^* + D F^*  \cdot \sigma\right]\\
\nonumber &=&-\kappa h_{F^*,\sigma},
\end{eqnarray}
therefore
$$ D R_\sigma[F^*] h_{F^*,\sigma}=0. $$ 

Now, suppose $h$ is an eigenvector of $D R_*[F^*]$ corresponding to the eigenvalue $\delta \ne \kappa$, hence, $h \ne h_{F^*,\sigma}$, then, since $E(\kappa) E(\delta)=0$, so is $D_F t_\sigma[F^*]h$, and $D_F S_{t_\sigma[F^*]}^\sigma h$. It follows from $(\ref{DR_action})$ that 
$$ D R_\sigma[F^*] h=\delta h.$$

Vice verse, suppose $h$ is an eigenvector of $D R_\sigma[F^*]$ corresponding to an eigenvalue $\delta \ne \kappa$, then,
$$D_F t_\sigma[F^*]h=-\kappa^{-1}\|h_{F^*,\sigma}\|_\cD^{-1}\|E(\kappa)DR_*[F^*]h\|_\cD,$$
and by (\ref{DR_action}) and a similar computation as above, for $a\in \R$,
\begin{eqnarray*}
DR_*[F^*](h+ah_{F^*,\sigma})
&  =&  a\kappa h_{F^*,\sigma}+DR_*[F^*]h \\
&  =&  a\kappa h_{F^*,\sigma}+\delta h-
\left(\Lambda_*^{-1} \cdot \left( D_F \left(S^\sigma_{t[F^*]}\right)^{-1} h \right)
\circ \Lambda_* \circ  F^* \right.\\
  &\phantom{=}&  \phantom{\delta h} \left.+ \left[ D F^*  \cdot \Lambda_*^{-1} \cdot
\left( D_F S_{t[F^*]}^\sigma h \right) \right] \circ \Lambda_* \right)\\
&  =&  a\kappa h_{F^*,\sigma}+\delta
h+\|h_{F^*,\sigma}\|_\cD^{-1}\|E(\kappa)DR_*[F^*]h\|_\cD h_{F^*,\sigma}.
\end{eqnarray*}
Let,
$$
a={{\|E(\kappa)DR_*[F^*]h\|_\cD} \over {\|h_{F^*,\sigma}\|_\cD(\delta-\kappa)}},
$$
then $h+a h_{F^*,\sigma}$ is an eigenvector of $DR_*[F^*]$ with eigenvalue $\delta$.

\end{proof}

\begin{lemma}\label{similar_spectra}

Suppose that there are $\beta$, $\varrho$, $\rho$, $\la_*$, $\mu_*$ and a function $s^* \in \cA^\beta_s(\rho)$  such that the operator $R_{EKW}$ is analytic and compact as maps from $\cF^{\beta,\rho}_\varrho(s^*)$  to $\cO_2(\cD)$, and 
$$R_{EKW}[F^*]=R_*[F^*]=F^*,$$
where $F^*$ is generated by $s^*$.

Then, there exists a neighborhood $\cB(F^*) \subset  \cF^{\beta,\rho}_\varrho(s^*)$, in which $R_*$ is analytic and compact, and
$${\rm spec}(D R_*[F^*]\arrowvert_{T_{F^*} \cB(F^*)} )= {\rm spec}(D R_{EKW}[F^*] \arrowvert_{T_{F^*} \cF^{\beta,\rho}_\varrho (s^*) } )\cup  \{1\}.$$
\end{lemma}

\begin{proof}
Let $\sigma^1_{0,0}$ and $\sigma^2_{0,0}$ be as in $(\ref{sigmas})$,  then 
\begin{eqnarray}
\nonumber S^{\sigma_{0,0}^1}_{\eps}(x,u)&=&((1+\eps)x,u), \quad h_{F,\sigma_{0,0}^1}=\pi_x F+DF\cdot (\pi_x,0), \\
\nonumber S^{\sigma_{0,0}^2}_{\eps}(x,u)&=&(x,(1+\eps)u), \quad  h_{F,\sigma_{0,0}^2}=\pi_u F+DF\cdot (0,\pi_u).
\end{eqnarray}

Now, notice, that the operator $R_{EKW}[F]$ can be written as 
$$R_{EKW}[F]=\Lambda_*^{-1} \circ \left( S^{\sigma_{0,0}^1}_{t_{EKW}[F]} \right)^{-1} \circ \left( S^{\sigma_{0,0}^2}_{r_{EKW}[F]} \right)^{-1} \circ F \circ F \circ  S^{\sigma_{0,0}^2}_{r_{EKW}[F]} \circ  S^{\sigma_{0,0}^1}_{t_{EKW}[F]} \circ \Lambda_*,$$
where 
$$t_{EKW}[F]={\pi_x F(F(0,0)) \over \lambda_*}-1,\quad r_{EKW}[F]={\pi_x F(F(0,0)) \over \mu_* \pi_x (F \circ F)_2(0,0)}-1={\lambda_* (1+t_{EKW}[F]) \over \mu_* \pi_x (F \circ F)_2(0,0)}-1,$$

Notice, that that $t_{EKW}[F]$, $r_{EKW}[F]$, and therefore the transformations $S^{\sigma_{0,0}^1}_{t_{EKW}[F]}$ and $S^{\sigma_{0,0}^2}_{r_{EKW}[F]}$, depend only on $P_{EKW}[F]$. Therefore, the maps $F \mapsto  S^{\sigma_{0,0}^1}_{t_{EKW}[F]}$ and $F \mapsto  S^{\sigma_{0,0}^2}_{r_{EKW}[F]}$ are analytic  (differentiable). In particular, by the continuity of $F \mapsto  S^{\sigma_{0,0}^1}_{t_{EKW}[F]}$ and $F \mapsto  S^{\sigma_{0,0}^2}_{r_{EKW}[F]}$, there exists a neighborhood $\cB(F^*) \subset \aF$, such that $R_*$ is compact in $\cB(F^*)$. In particular, both $D R_*[F^*]$ and $D  R_{EKW}[F^*]$ exist, and are compact linear operators.

For any $F \in \cB(F^*)$ and $h \in T_{F^*}\aF$,
\begin{eqnarray}
\nonumber D \!\!\!\!\!\!\!\!&\!\!\!\!\!\!\!\!\!\!&\!\!\!\!\!\!\!\! R_{EKW}[F] h=D R_*[F] h\\
\nonumber &+&\Lambda_*^{-1}  \left(D_F \left(S^{\sigma_{0,0}^1}_{t_{EKW}[F]} \right)^{-1} h\right) \circ  \left( S^{\sigma_{0,0}^2}_{r_{EKW}[F]} \right)^{-1} \circ F \circ F \circ  S^{\sigma_{0,0}^2}_{r_{EKW}[F]} \circ  S^{\sigma_{0,0}^1}_{t_{EKW}[F]} \circ  \Lambda_*\\
\nonumber &+&\Lambda_*^{-1} \left[D \left(\left( S^{\sigma_{0,0}^1}_{t_{EKW}[F]} \right)^{-1} \circ \left( S^{\sigma_{0,0}^2}_{r_{EKW}[F]} \right)^{-1} \circ F \circ F \circ  S^{\sigma_{0,0}^2}_{r_{EKW}[F]} \right) \cdot  \left(D_F S^{\sigma_{0,0}^1}_{t_{EKW}[F]}h\right)\right] \circ \Lambda_*\\
\nonumber  &+&\Lambda_*^{-1} \cdot  D \left( S^{\sigma_{0,0}^1}_{t_{EKW}[F]} \right)^{-1} \cdot \left(D_F \left( S^{\sigma_{0,0}^2}_{r_{EKW}[F]}\right)^{-1} h\right) \circ F \circ F \circ  S^{\sigma_{0,0}^2}_{r_{EKW}[F]} \circ  S^{\sigma_{0,0}^1}_{t_{EKW}[F]} \circ \Lambda_*\\
\nonumber &+&\Lambda_*^{-1}\left[ D \left(   \left( S^{\sigma_{0,0}^1}_{t_{EKW}[F]} \right)^{-1} \circ \left( S^{\sigma_{0,0}^2}_{r_{EKW}[F]} \right)^{-1} \circ F \circ F \right) \cdot \left(D_F S^{\sigma_{0,0}^2}_{r_{EKW}[F]} h\right)\right] \circ  S^{\sigma_{0,0}^1}_{t_{EKW}[F]} \circ   \Lambda_*\\
\nonumber  &=&D R_*[F] h-\left(D_F t_{EKW}[F] h \right) \Lambda_*^{-1} \circ \sigma_{0,0}^1 \circ  \left( S^{\sigma_{0,0}^2}_{r_{EKW}[F]} \right)^{-1} \circ  F \circ F \circ  S^{\sigma_{0,0}^2}_{r_{EKW}[F]} \circ  S^{\sigma_{0,0}^1}_{t_{EKW}[F]}  \circ \Lambda_*\\
\nonumber &+&\left(D_F t_{EKW}[F] h \right)  \Lambda_*^{-1} \left[D \left(\left( S^{\sigma_{0,0}^1}_{t_{EKW}[F]} \right)^{-1} \circ \left( S^{\sigma_{0,0}^2}_{r_{EKW}[F]} \right)^{-1} \circ  F \circ F \circ  S^{\sigma_{0,0}^2}_{r_{EKW}[F]}\right)  \circ \sigma_{0,0}^1 \right] \circ \Lambda_*\\
\nonumber  &-&\left(D_F r_{EKW}[F] h \right) \Lambda_*^{-1} \cdot D \left( S^{\sigma_{0,0}^1}_{t_{EKW}[F]} \right)^{-1}  \circ \sigma_{0,0}^2 \circ F \circ F \circ  S^{\sigma_{0,0}^2}_{r_{EKW}[F]} \circ  S^{\sigma_{0,0}^1}_{t_{EKW}[F]} \circ \Lambda_*\\
\nonumber &+&\left(D_F r_{EKW}[F] h \right)  \Lambda_*^{-1} \left[D \left( \left( S^{\sigma_{0,0}^1}_{t_{EKW}[F]} \right)^{-1} \circ \left( S^{\sigma_{0,0}^2}_{r_{EKW}[F]} \right)^{-1} \circ F \circ F \right)  \circ \sigma_{0,0}^2 \right]  \circ  S^{\sigma_{0,0}^1}_{t_{EKW}[F]} \circ \Lambda_*.
\end{eqnarray}

Specifically, if $F=F^*$, then (cf. $(\ref{derivation})$)
\begin{eqnarray}\label{R0REKW_eigenvectors}
\nonumber D R_{EKW}[F^*] h&=&D R_*[F^*] h+\left(D_F t_{EKW}[F^*] h \right) h_{F^*,\sigma_{0,0}^1}\\
&\phantom{=}& \phantom{ D R_*[F^*] h  } +\left(D_F r_{EKW}[F^*] h \right) h_{F^*,\sigma_{0,0}^2}.
\end{eqnarray}

Next,
\begin{eqnarray}
\nonumber D_F S^{\sigma_{0,0}^1}_{t_{EKW}[F]}h&=&(D_F  t_{EKW}[F] h \pi_x,0),\\
\nonumber D_F S^{\sigma_{0,0}^2}_{r_{EKW}[F]}h&=&(0,D_F  r_{EKW}[F] h \pi_u),\\
\nonumber D_F t_{EKW}[F] h&=&{\pi_x DP_{EKW}[F]h(0,0) \over \lambda_*},\\
\nonumber D_F r_{EKW}[F] h&=&{\lambda_* D_F t_{EKW}[F] h \over \mu_*  \pi_x (F \circ F)_2(0,0) } -{\lambda_* \pi_x \left(DP_{EKW}[F]h\right)_2(0,0) \over \mu_*  \left(\pi_x (F \circ F)_2(0,0) \right)^2}.  
\end{eqnarray}

If $h=h_{F^*,\sigma_{0,0}^1}$, then
\begin{eqnarray}
\nonumber DP_{EKW}[F]h(x,u)&=&\left(-\pi_x P_{EKW}[F](x,u)+\pi_x P_{EKW}[F]_1(x,u) x,\right.\\
\nonumber &\phantom{=}& \left. \phantom{-.}  \pi_u P_{EKW}[F]_1(x,u) x \right),\\
\nonumber \pi_x DP_{EKW}[F]h(0,0)&=&-\pi_x P_{EKW}[F](0,0)=-\lambda_*,\\
\nonumber D_F t_{EKW}[F] h&=&-1,\\
\nonumber D_F r_{EKW}[F] h&=& -{\lambda_* \over \mu_*  \pi_x (F \circ F)_2(0,0) } \\
\nonumber &\phantom{=}& -{\lambda_* \left( -\pi_x P_{EKW}[F]_2(0,0)+\pi_x P_{EKW}[F]_{1,2}(0,0) 0 \right) \over \mu_*  \left(\pi_x (F \circ F)_2(0,0) \right)^2}=0,\\
\nonumber  D_F S^{\sigma_{0,0}^1}_{t_{EKW}[F]}h&=&(-\pi_x,0),\\
\nonumber  D_F \left(S^{\sigma_{0,0}^1}_{t_{EKW}[F]}\right)^{-1}h&=&(\pi_x,0).
\end{eqnarray}


Similarly, if $h=h_{F^*,\sigma_{0,0}^2}$, then
\begin{eqnarray}
\nonumber DP_{EKW}[F]h(x,u)&=&\left(\pi_x P_{EKW}[F]_2(x,u) u,\right.\\
\nonumber &\phantom{=}&\!\!\! \left.-\pi_u P_{EKW}[F](x,u)+\pi_u P_{EKW}[F]_2(x,u) u \right),\\
\nonumber \pi_x DP_{EKW}[F]h(0,0)&=&0,\\
\nonumber  D_F t_{EKW}[F] h&=&0,\\
\nonumber  D_F r_{EKW}[F] h&=&-1,\\
\nonumber D_F S^{\sigma_{0,0}^2}_{r_{EKW}[F]}h&=&(0,-\pi_u),\\
\nonumber D_F \left(S^{\sigma_{0,0}^2}_{r_{EKW}[F]}\right)^{-1}h&=&(0,\pi_u).
\end{eqnarray}

Therefore, if $h=h_{F^*,\sigma_{0,0}^1}$, we get
\begin{eqnarray}
\nonumber D R_{EKW}[F^*] h\!\!&\!\!=\!\!&\!\! \Lambda_*^{-1} D P_{EKW}[F^*] h \circ \Lambda_*+\Lambda_*^{-1}  \pi_x F \circ F \circ \Lambda_*\\
\nonumber \!\!&\!\!\phantom{=}\!\!&\!\! \phantom{\Lambda_*^{-1} D P_{EKW}[F^*] h \circ \Lambda_*}+\Lambda_*^{-1} \left[D \left(F \circ F \right)  \cdot (-\pi_x,0) \right]\circ \Lambda_*\\
\nonumber  \!\!&\!\!\phantom{=}\!\!&\!\! \phantom{\Lambda_*^{-1} D P_{EKW}[F^*] h \circ \Lambda_*}  +\left( D_F r_{EKW}[F^*] h \right) h_{F^*,\sigma_{0,0}^2}.\\
\nonumber \!\!&\!\!=\!\!&\!\!\Lambda_*^{-1} \left[ D P_{EKW}[F^*] h +\pi_x P_{EKW}[F^*] \right. \\
\nonumber \!\!&\!\!\phantom{=}\!\!&\!\!\phantom{\Lambda_*^{-1} \left[ D P_{EKW}[F^*] h \right. } - \left. \left( \pi_x P_{EKW}[F]_1 \pi_x,  \pi_u P_{EKW}[F]_1 \pi_x\right)  \right]\circ \Lambda_*\\
\nonumber \!\!&\!\!\phantom{=}\!\!&\!\!\phantom{\Lambda_*^{-1}\left[ D P_{EKW}[F^*] h\right.} +0\\
\nonumber \!\!&\!\!=\!\!&\!\! 0.
\end{eqnarray}

If $h=h_{F^*,\sigma_{0,0}^2}$, then
\begin{eqnarray}
\nonumber D R_{EKW}[F^*]  h\!\!&\!\!=\!\!&\!\!\Lambda_*^{-1} D P_{EKW}[F^*] h \circ \Lambda_*+\Lambda_*^{-1}  \pi_u F \circ F \circ \Lambda_*\\
\nonumber \!\!&\!\!\phantom{=}\!\!&\!\! \phantom{\Lambda_*^{-1} D P_{EKW}[F^*] h \circ \Lambda_*} +\Lambda_*^{-1} \left[D \left(F \circ F \right)  \cdot (0,-\pi_u) \right]\circ \Lambda_*\\
\nonumber  \!\!&\!\!\phantom{=}\!\!& \!\!\phantom{\Lambda_*^{-1} D P_{EKW}[F^*] h \circ \Lambda_*}  +\left(D_F t_{EKW}[F^*] h \right) h_{F^*,\sigma_{0,0}^1}\\
\nonumber \!\!&\!\!=\!\!&\!\!\Lambda_*^{-1} \left[D P_{EKW}[F^*] h  + \pi_u P_{EKW}[F^*] \right.\\
\nonumber  \!\!&\!\!\phantom{=}\!\!&\!\!\phantom{L \Lambda_*^{-1} D P_{EKW}[F^*] h } -\left. \left( \pi_x P_{EKW}[F]_2 \pi_u,  \pi_u P_{EKW}[F]_2 \pi_u\right) \right] \circ \Lambda_*\\
\nonumber   \!\!&\!\!\phantom{=}\!\!&\!\!\phantom{L \Lambda_*^{-1} D P_{EKW}[F^*] h } +0\\
\nonumber \!\!&\!\!=\!\!&\!\! 0.
\end{eqnarray}

If $h$ is an eigenvector of $D R_*[F^*]$ associated with a non-zero eigenvalue $\kappa$, $h  \ne h_{F^*,\sigma_{0,0}^1}$, and  $h  \ne h_{F^*,\sigma_{0,0}^2}$, then for any constant $a$ and $b$
\begin{eqnarray}
\nonumber &&D R_{EKW}[F^*](h+a  h_{F^*,\sigma_{0,0}^1}+ b h_{F^*,\sigma_{0,0}^2} ) =\\
\nonumber &&=D R_*[F^*] h +a  h_{F^*,\sigma_{0,0}^1}+b  h_{F^*,\sigma_{0,0}^2}+\\
\nonumber  &&\phantom{=D R_*[F^*] h}+\left(D_F t_{EKW}[F^*] \left(h+ a  h_{F^*,\sigma_{0,0}^1}+ b h_{F^*,\sigma_{0,0}^2} \right) \right) h_{F^*,\sigma_{0,0}^1}\\
\nonumber  &&\phantom{=D R_*[F^*] h}+\left(D_F r_{EKW}[F^*] \left(h+ a  h_{F^*,\sigma_{0,0}^1}+ b h_{F^*,\sigma_{0,0}^2} \right) \right) h_{F^*,\sigma_{0,0}^2}\\
\nonumber &&=\kappa h +a  h_{F^*,\sigma_{0,0}^1}+b  h_{F^*,\sigma_{0,0}^2}+\\
\nonumber  &&\phantom{=\kappa h}+\left(D_F t_{EKW}[F^*] \left(h+ b h_{F^*,\sigma_{0,0}^2} \right) \right) h_{F^*,\sigma_{0,0}^1} -a h_{F^*,\sigma_{0,0}^1} \\
\nonumber  &&\phantom{=\kappa h}+\left(D_F r_{EKW}[F^*] \left(h+ a  h_{F^*,\sigma_{0,0}^1} \right) \right) h_{F^*,\sigma_{0,0}^2}- b h_{F^*,\sigma_{0,0}^2}\\
\nonumber &&=\kappa h +\kappa_1 h_{F^*,\sigma_{0,0}^1} +\kappa_2 h_{F^*,\sigma_{0,0}^2},
\end{eqnarray}
where 
$$\kappa_1[h]=D_F t_{EKW}[F^*]h, \quad \kappa_2[h]=D_F r_{EKW}[F^*]h,$$
and we see, that if $a[h]={\kappa_1 / \kappa}$ and $b[h]={\kappa_2 / \kappa}$, then 
$$h+a  h_{F^*,\sigma_{0,0}^1}+ b h_{F^*,\sigma_{0,0}^2}$$
is an eigenvector for $D R_{EKW}[F^*]$ with the eigenvalues $\kappa$. 

On the other hand, if $h$ is en eigenvector of  $D R_{EKW}[F^*]$ associated with the eigenvalue $\kappa \ne 1$, then 
$$h-a h_{F^*,\sigma_{0,0}^1}- b h_{F^*,\sigma_{0,0}^2}$$
is an eigenvector of $D R_*[F^*]$ associated with $\kappa$.

\end{proof}

\bigskip \section{Strong contraction on the stable manifold}

\begin{lemma}\label{lambda_lemma}
Suppose that $\beta$, $\varrho$ and $\rho$ are such that the operator 
$$\cR_*[s]={1 \over \mu_*} {\cP}_{EKW}[s] \circ \lambda_*$$
has a fixed point $s^* \in \cB_\varrho \subset \cA_s^{\beta}(\rho)$, and $\cR_*$ is analytic and compact as a map from $\cB_\varrho$  to $\cA_s^{\beta}(\rho)$. 

Then, the number $\lambda_*$ is an eigenvalue of $D \cR_*[s^*]$, and the eigenspace of $\lambda_*$ contains the eigenvector 
\begin{equation}\label{psi_eigenvector}
\psi_{s^*}(x,y)=s_1^*(x,y) x^2+s_2^*(x,y) y^2+2 s^*(x,y) y.
\end{equation}
\end{lemma}

\begin{proof}
Consider the coordinate transformation $\ref{coords})$, 
\begin{eqnarray}
 \nonumber S_\epsilon(x,u)&=&\left(x+\epsilon x^2,{u \over 1+2 \epsilon x}\right)\\
\label{coord} &=&(x,u) +\epsilon \sigma^1_{1,0}(x,u)-2 \epsilon \sigma^2_{1,0}(x,u)+O(\epsilon^2),\\
\label{inv_coord} S_\epsilon^{-1}(y,v)&=&\left({\sqrt{1+4\epsilon y}-1 \over 2 \epsilon}, v \sqrt{1+4 \epsilon y}  \right),
\end{eqnarray}
for real $\epsilon$, $|\epsilon|<4/(\rho+|\beta|)$ (recall Definition $\ref{B_space}$).

Let $s \in   \cA_s^\beta(\rho)$ be the generating function for some $F$, then the following demonstrates that  $S_\epsilon^{-1} \circ F \circ S_\epsilon$ is reversible, area-preserving and generated by
$$\hat{s}(x,y)=s(x+\epsilon x^2, y+ \epsilon y^2) (1+2 \epsilon y):$$ 

\begin{eqnarray}
\nonumber \left({x  \atop  -s(y+\epsilon y^2,x+\epsilon x^2)(1+2 \epsilon x)} \right)  &{{ \mbox{{\small \it  $S_\epsilon$}} \atop \mapsto} \atop \phantom{\mbox{\tiny .}}}&  \left({x+\epsilon x^2  \atop  -s(y+\epsilon y^2,x+\epsilon x^2)} \right)\\
\nonumber =\left({x'  \atop  -s(y',x')} \right) &  {{ \mbox{{\small \it  F}} \atop \mapsto} \atop \phantom{\mbox{\tiny .}}} & \left({y' \atop s(x',y') }\right)\\
\nonumber = \left({y+\epsilon y^2 \atop s(x+\epsilon x^2,y+\epsilon y^2) }\right) &  {{ \mbox{{\small \it  $S_\epsilon^{-1}$}} \atop \mapsto} \atop \phantom{\mbox{\tiny .}}} & \left({y \atop s(x+\epsilon x^2,y+\epsilon y^2) (1+2 \epsilon y) }\right).  
\end{eqnarray}

Next,
$$\hat{s}(x,y)=s(x,y)+\epsilon s_1(x,y) x^2+\epsilon s_2(x,y) y^2+\epsilon 2 s (x,y) y +O(\epsilon^2).$$
We will demonstrate that 
$$\psi_{s^*}(x,y)=s_1^{*}(x,y) x^2+s_2^{*}(x,y) y^2 +2 s^{*} (x,y) y.$$
is an eigenvector of $D \cR_*[s^{*}]$ of the eigenvalue $\lambda_*$. Notice, that 
$$\partial_1 \psi_s=\partial_1 \psi_s \circ I, \quad I(x,y)=(y,x),$$
and therefore, the function $s+\epsilon \psi_s \in \cA_s^\beta(\rho)$.

Consider the midpoint equation 
\begin{eqnarray}
\nonumber 0=O(\epsilon^2)\!&\!+\!&\!s(x,Z(x,y)+\epsilon DZ[s] \psi_s(x,y))+s(y,Z(x,y)+\epsilon DZ[s] \psi_s(x,y))\\
\nonumber \!&\!+\!&\! \epsilon \psi_s(x,Z(x,y))+\epsilon \psi_s(y,Z(x,y))
\end{eqnarray}
for the generating function $s+\epsilon \psi_s$. We get that
$$D Z[s] \psi_s(x,y)=- { \psi_s(x,Z(x,y))+\psi_s(y,Z(x,y))  \over  s_2(x,Z(x,y))+s_2(y,Z(x,y)) },$$
and
\begin{eqnarray}
\nonumber D \cP_{EKW} \psi_s(x,y)&=&s_1(Z(x,y),y) D Z[s] \psi_s(x,y) + \psi_s(Z(x,y),y)\\
\nonumber &=&- 2 s_1(Z(x,y),y) {s(x,Z(x,y)) Z +s(y,Z(x,y)) Z \over  s_2(x,Z(x,y))+s_2(y,Z(x,y)) }\\ 
\nonumber &\phantom{=}&-  s_1(Z(x,y),y) { s_2(x,Z(x,y)) Z(x,y)^2 +s_2(y,Z(x,y)) Z(x,y)^2 \over  s_2(x,Z(x,y))+s_2(y,Z(x,y)) }\\
\nonumber &\phantom{=}&+s_1(Z(x,y),y) Z(x,y)^2\\
\nonumber &\phantom{=}&- s_1(Z(x,y),y) { s_1(y,Z(x,y)) y^2  \over  s_2(x,Z(x,y))+s_2(y,Z(x,y)) } +s_2(Z(x,y),y) y^2\\
\nonumber &\phantom{=}&-s_1(Z(x,y),y) { s_1(x,Z(x,y)) x^2  \over  s_2(x,Z(x,y))+s_2(y,Z(x,y)) }\\
\nonumber &\phantom{=}&+2 s(Z(x,y),y) y 
\end{eqnarray} 

The terms on line 2 add up to zero (the numerator is equal to zero because of the midpoint equation), and so do those on lines 3 and 4.  We can also use the equalities 
\begin{eqnarray}
\nonumber s_2(x,Z(x,y))+s_2(y,Z(x,y))&=&-{ s_1(y,Z(x,y)) \over Z_2(x,y) }\\ 
\nonumber \partial_2 \cP_{EKW}[s](x,y)&=&s_2(Z(x,y),y) +s_1 (Z(x,y),y) Z_2(x,y)
\end{eqnarray} 
(the first one being the midpoint equation differentiated with respect to $y$) to reduce the 5-th line to 
$$ \partial_2 \cP_{EKW}[s](x,y) y^2.$$
The 6-th line reduces to 
$$ \partial_1 \cP_{EKW}[s](x,y) x^2$$ 
after we use the midpoint equation differentiated with respect to $x$:
$$ s_2(x,Z(x,y)+s_2(y,Z(x,y)=-{ s_1(x,Z(x,y)) \over Z_1(x,y) }.$$

To summarize,
\begin{eqnarray}
\nonumber D \cP_{EKW} \psi_s(x,y)&=&
\partial_1 \cP_{EKW}[s](x,y) x^2+ \partial_2 \cP_{EKW}[s](x,y) y^2+2\cP_{EKW}[s](x,y) y \\
\nonumber &=&\psi_{\cP_{EKW}[s]}(x,y).
\end{eqnarray}

Finally, we use the fact that
$$ \lambda_* \partial_i \cP_{EKW}[s](\lambda_* x, \lambda_* y)= \partial_i \left( {\cP}[s](\lambda_* x, \lambda_* y) \right)$$ 
to get
$$ D \cR_*[s^*] \psi_{s^*}=\lambda_* \psi_{s^*}.$$

\end{proof}

The Lemma below, whose elementary  proof we will omit, shows that $\lambda_*$ is also in the spectrum of $D R_*[F^*]$:

\begin{lemma}\label{same_spectra_new}
Suppose that $\beta$, $\varrho$ and $\rho$ are such that $s^* \in \cA_s^{\beta}(\rho)$ is a fixed point of $\cR_*$, and the operator $\cR_*$ is analytic and compact as a map from $\Bstar$  to $\cA_s^{\beta}(\rho)$.  Also, suppose that the map $\cI$, described in Remark $\ref{mapI}$, is well-defined and analytic on $\Bstar$, and that it has an  analytic inverse $\cI^{-1}$  on $\cI(\cB_\varrho(s^{*}))$. Then, 
$${\rm spec}\left(\left(D R_*[F^*]\right)\arrowvert_{T_{F^*} \aaF} \right) = {\rm spec } \left(D \cR_* [s^*] \right).$$
in particular,
$$\lambda_* \in  {\rm spec } \left(D R_* [F^*] \right).$$ 
\end{lemma}

At the same time, it is straightforward to see that  the spectra of $D R_{EKW}[F_{EKW}]\arrowvert_{T_{F_{EKW}} \aaF}$ and $D \cR_{EKW}[s^{EKW}]$ are identical.

\begin{lemma}\label{same_spectra}
Suppose that $\beta$, $\varrho$ and $\rho$ are such that $s^* \in \cA_s^{\beta}(\rho)$, and the operator $\cR_{EKW}$ is analytic and compact as a map from $\Bstar$  to $\cA_s^{\beta}(\rho)$.  Also, suppose that the map $\cI$, described in Remark $\ref{mapI}$, is well-defined and analytic on $\Bstar$, and that it has an  analytic inverse $\cI^{-1}$  on $\cI(\cB_\varrho(s^*))$. Then, 

$${\rm spec}\left(\left(D R_{EKW}[F^*]\right)\arrowvert_{T_{F^*} \aaF} \right) = {\rm spec } \left(D \cR_{EKW} [s^*] \right),$$
in particular,
$$\lambda_* \in  {\rm spec } \left(D \cR_{EKW} [s^*] \right).$$ 
\end{lemma}

The convergence rate in the stable manifold of the renormalization operator plays a crucial role in demonstrating rigidity. It turns out that {\it the eigenvalue $\lambda_*$ is the largest eigenvalues in the stable subspace of $D R_{EKW}[F^*]$}, or equivalently $D \cR_{EKW}[s^{*}]$. However, it's value $|\lambda_*| \approx 0.2488$ is not small enough to ensure rigidity. At the same time, the eigenspace of the eigenvalue $\lambda_*$ is, in the terminology  of the renormalization theory, {\it irrelevant} to dynamics (the associated eigenvector is generated by a coordinate transformation). We, therefore, would like to  eliminate this eigenvalue via an appropriate coordinate change, as described above.

However, first we would like to identify the eigenvector corresponding to the eigenvalue $\lambda_*$ for the operator $\cR_{EKW}$. This vector turns out to be different from $\psi_{s^*}$.

\begin{lemma}\label{new_lambda_lemma}
Suppose that $\beta$, $\varrho$ and $\rho$ are such that the operator 
$\cR_{EKW}$ has a fixed point $s^* \in \cA_s^{\beta}(\rho)$, and $\cR_{EKW}$ is analytic and compact as a map from $\Bstar$  to $\cA_s^{\beta}(\rho)$.  Also, suppose that the map $\cI$, described in Remark $\ref{mapI}$, is well-defined and analytic on $\Bstar$, and that it has an  analytic inverse $\cI^{-1}$  on $\cI(\cB_\varrho(s^*))$.

Then, the number $\lambda_*$ is an eigenvalue of $D \cR_{EKW}[s^*]$, and the eigenspace of $\lambda_*$ contains the eigenvector 
\begin{equation}\label{new_psi_eigenvector}
\psi^{EKW}_{s^*}(x,y)=\psi_{s^*}+\tilde{\psi},
\end{equation}
where
$$\tilde{\psi}=s^*-(s_1^*(x,y)x+s_2^*(x,y) y).$$
\end{lemma}

\begin{proof}
Notice, that $\tilde{\psi}$ is of the form
$$\tilde{\psi}(x,y)=\psi_u-\psi_x,$$
where 
$$\psi_x(x,y)=s^*_1(x,y) x+s^*_2(x,y) y$$ 
is the eigenvector of $D \cR_*[s^*]$ corresponding to the rescaling of the variables $x$ and $y$, while 
$$\psi_u(x,y)=s^*(x,y)$$ 
is the eigenvector corresponding to the rescaling of $s$. $\psi_x(x,y)$ and $\psi_u(x,y)$ correspond to the eigenvectors $h_{F^*,\sigma^1_{0,0}}$ and $h_{F^*,\sigma^2_{0,0}}$, respectively, of $D R_{0}[F^*]$. 

Recall, that $h_{F^*,\sigma^1_{0,0}}$ and $h_{F^*,\sigma^2_{0,0}}$ are eigenvectors of $D R_{0}[F^*]$, with eigenvalue $1$, and eigenvectors of $D R_{EKW}[F^*]$ with eigenvalue $0$. 

By Lemma \ref{lambda_lemma} $\psi_{s^*}$ is an eigenvector of $D\cR_*$, the corresponding eigenvector of $DR_*$ is $h_{F^*,\sigma^1_{1,0}-2\sigma^2_{1,0}}$. Thus, $\psi_{s^*}+\tilde \psi$ corresponds to the vector 
\begin{equation}
h_{\lambda_*}^{EKW} := h_{F^*,\sigma^1_{1,0}-2\sigma^2_{1,0}} - h_{F^*,\sigma^1_{0,0}}+ h_{F^*,\sigma^2_{0,0}}.
\end{equation}
 
To finish the proof, it suffices to prove that $$DR_{EKW}h_{\lambda_*}^{EKW} = \lambda_* h_{\lambda_*}^{EKW}.$$

By (\ref{R0REKW_eigenvectors})
\begin{eqnarray}
\nonumber D R_{EKW}[F^*] h_{\lambda_*}^{EKW} &=& D R_{EKW}[F^*] h_{F^*,\sigma^1_{1,0}-2\sigma^2_{1,0}} \\
\nonumber & = &D R_*[F^*] h_{F^*,\sigma^1_{1,0}-2\sigma^2_{1,0}} \\
\nonumber &+& \left(D_F t_{EKW}[F^*] h_{F^*,\sigma^1_{1,0}-2\sigma^2_{1,0}}\right) h_{F^*,\sigma_{0,0}^1} \\
\nonumber &+& \left(D_F r_{EKW}[F^*] h_{F^*,\sigma^1_{1,0}-2\sigma^2_{1,0}} \right) h_{F^*,\sigma_{0,0}^2} \\
\nonumber & = &\lambda_* h_{F^*,\sigma^1_{1,0}-2\sigma^2_{1,0}} \\
\nonumber &+& \left(D_F t_{EKW}[F^*] h_{F^*,\sigma^1_{1,0}-2\sigma^2_{1,0}}\right) h_{F^*,\sigma_{0,0}^1} \\
\nonumber &+& \left(D_F r_{EKW}[F^*] h_{F^*,\sigma^1_{1,0}-2\sigma^2_{1,0}} \right) h_{F^*,\sigma_{0,0}^2}
\end{eqnarray}
The result follows if $$D_F t_{EKW}[F^*] h_{F^*,\sigma^1_{1,0}-2\sigma^2_{1,0}}=-\lambda_*$$ and 
$$D_F r_{EKW}[F^*] h_{F^*,\sigma^1_{1,0}-2\sigma^2_{1,0}}=\lambda_*.$$

Indeed, as in the proof of Lemma \ref{similar_spectra}. If $h=h_{F^*,\sigma^1_{1,0}}$, then 
\begin{eqnarray}
\nonumber DP_{EKW}[F^*]h(x,u)&=&\left(-(\pi_x P_{EKW}[F^*](x,u))^2+\pi_x P_{EKW}[F^*]_1(x,u) x^2, \right.\\
\nonumber & & \left.\quad \pi_u P_{EKW}[F^*]_1(x,u) x^2 \right),\\
\nonumber \pi_x DP_{EKW}[F^*]h(0,0)&=&-(\pi_x P_{EKW}[F^*](0,0))^2=-\lambda_*^2,\\
\nonumber D_F t_{EKW}[F^*] h&=&-\lambda_*\\
\nonumber D_F r_{EKW}[F^*] h&=& {\lambda_*^2 \over \mu_*  \pi_x (F^* \circ F^*)_2(0,0) } \\ 
\nonumber &+& {\lambda_* \left( -2\pi_x P_{EKW}[F^*](0,0)\pi_x P_{EKW}[F^*]_2(0,0)\right.}\\
\nonumber & & \quad +{\left. \pi_x P_{EKW}[F^*]_{1,2}(0,0) 0^2 \right) \over \mu_*  \left(\pi_x (F^* \circ F^*)_2(0,0) \right)^2}\\
\nonumber &=& -\lambda_*+2\pi_x P_{EKW}[F^*](0,0) = \lambda_*
\end{eqnarray}
If $h=h_{F^*,\sigma^2_{1,0}}$, then
\begin{eqnarray}
\nonumber DP_{EKW}[F^*]h(x,u)&=&\left(\pi_x P_{EKW}[F^*]_2(x,u) xu, \right.\\
\nonumber & & \quad -\pi_x P_{EKW}[F^*](x,u)\pi_u P_{EKW}[F^*](x,u) \\
\nonumber & & \left.\quad+\pi_u P_{EKW}[F^*]_2(x,u) xu \right),\\
\nonumber \pi_x DP_{EKW}[F^*]h(0,0)&=& 0\\
\nonumber D_F t_{EKW}[F^*] h&=& 0\\
\nonumber D_F r_{EKW}[F^*] h&=& 0 + {\lambda_* \left(\pi_x P_{EKW}[F^*]_{2,2}(0,0) 0+\pi_x P_{EKW}[F^*]_{2}(0,0)0  \right) \over \mu_*  \left(\pi_x (F^* \circ F^*)_2(0,0) \right)^2} = 0.
\end{eqnarray}
\end{proof}

\begin{definition}\label{strong_renorm}
Suppose $s^*$ is a fixed point of the operator $\cR_*$ (or, equivalently, $\cR_{EKW}$).  Set, formally,
$$\nonumber  \cP[s](x,y)=( 1+ 2 t y) s(G(\xi_t(x,y))), \quad {\rm and} \quad \cR[s]=\mu^{-1} \cP[s] \circ \lambda,$$
where 
\begin{eqnarray}
\nonumber 0&=&s(x, Z(x,y) )+s( y, Z(x,y) ),\\
\nonumber t&=&-{1 \over \lambda_* \|  \psi^{EKW}_{s^*} \|_\rho}\|E(\lambda_*)(\cR_{EKW}[s]-s^{*})\|,\\
\label{new_lambda} 0 &=&\cP[s](\lambda,0),\\
\label{new_mu} \mu &=& \lambda \partial_1 \cP[s](\lambda,0),\\
\label{new_t} \xi_t(x,y)&=&(x+tx^2,y+ty^2),
\end{eqnarray}
$\psi^{EKW}_{s^*}$ is as in $(\ref{new_psi_eigenvector})$, $G$ as in (\ref{G_def}), and $E$ is the Riesz projection for the operator $D \cR_{EKW}[s^{*}].$
\end{definition}

We will quote a version of a lemma from $\cite{Gai4}$ which we will require to demonstrate analyticity and compactness of the operator $\cR$. The proof of the Lemma is computer-assisted. Notice, the parameters that enter the Lemma are different from those used in  $\cite{Gai4}$. As before, the reported numbers are representable on a computer.

\begin{lemma}\label{MC_lemma}
For all $s \in \cB_R(s^0)$, where
$$R=5.47321968732772541 \times 10^{-3},$$
and $s^0$ is as in Theorem $\ref{MyTheorem}$, the prerenormalization $\cP_{EKW}[s]$ is well-defined and analytic function on the set 
$$\cD_r \equiv \cD_r(0)=\{(x,y) \in \field{C}^2: |x|<r, |y|<r\}, \quad r=0.51853174082497335,$$
with  
$$\| Z\|_r \le 1.63160151494042404.$$
\end{lemma}

We will now demonstrate analyticity and compactness of the modified renormalization operator in a functional space, different from that used in \cite{EKW2}, specifically, in the space $\cA_s(1.75)$. It is in this space that  we will eventually compute a bound on the spectral radius of the action of the modified renormalization operator on infinitely renormalizable maps.

\begin{prop}\label{R_analyticity}
There exists a polynomial $s_0  \subset \cB_{R}(s^0) \subset \cA_s(1.75)$, where $R$ and $s^0$ are as in Lemma $\ref{MC_lemma}$, such that the operator $\cR$ is well-defined, analytic and compact as a map from $\cB_{\varrho_0}(s_0)$, $\varrho_0=5.79833984375 \times 10^{-4}$, to $\cA_s(1.75)$, if  $\cB_{\varrho_0}(s_0) \subset \cB_{R}(s^0)$ contains the fixed point $s^*$.
\end{prop}

\begin{proof}
The polynomial $s_0$ has been computed as a high order numerical approximation of a fixed point $s^*$ of $\cR$.

First, we get a bound on $t$ for all $s \in \cB_\delta(s_0)$:
\begin{eqnarray}
\nonumber |t|&=&{1 \over |\lambda_*| \|\psi^{EKW}_{s^*}\|_\rho} \|E(\lambda_*)(\cR_{EKW}[s]-s^*)\|_\rho\\
\nonumber &\le & {1 \over |\lambda_+|  \psi^{EKW}_{s^*}\|_\rho} \|\cR_{EKW}[s]-s^*\|_\rho.
\end{eqnarray}
We estimate  the right hand side rigorously on the computer and obtain
\begin{equation}\label{t[F]_bound}
|t| \le 2.1095979213715 \times 10^{-6}.
\end{equation}
The condition of the hypothesis that $s^* \in \cB_\delta(s_0)$ is specifically required to be able to compute this estimate.

Notice that according to Definition $\ref{strong_renorm}$ and  Theorem $\ref{MyTheorem}$, the maps $s \mapsto t$ and, hence, $s \mapsto \xi_t$ are analytic on a larger neighborhood $\cB_R(s^0)$ of analyticity of $\cR_{EKW}$. According to  Theorem $\ref{MyTheorem}$ and  Lemma $\ref{MC_lemma}$, the prerenormalization $\cP_{EKW}$ is also analytic as a map from $\cB_R(s^0)$ to $\cA_s(r)$, $r=0.516235055482147608$. We verify that for all $s \in \cB_\delta(s_0)$ and $t$ as in $(\ref{t[F]_bound})$ the following holds:
\begin{equation}\label{inclusion}
\{\xi_{t}(x, y): (x,y) \in \cD_{r'}\} \Subset \cD_{r}, \quad r'= |\lambda_-|\rho,
\end{equation}
where $\lambda_-= -0.27569580078125$ is the lower bound from Theorem $\ref{MyTheorem}$. Furthermore,
$$1>2 |t| \rho$$
with $t$ as in $(\ref{t[F]_bound})$. Therefore, the map $s \mapsto \cP[s]$ is analytic on $\cB_\delta(s_0)$. 
 

Since the inclusion of sets $(\ref{inclusion})$ is compact, $\cR[s]$ has an analytic extension to a neighborhood of $\cD_{1.75}$, $\cR[s] \in \cA_s(\rho')$, $\rho'>1.75$. Compactness of the map $s \mapsto \cR[s]$ now follows from the fact that the inclusions of spaces  $\cA_s(\rho') \subset \cA_s(\rho)$ is  compact.   
\end{proof}

\bigskip

Recall, that according to Lemma $\ref{same_spectra_new}$,  $\lambda_*$ is an eigenvalue of $D R_*[F^*]$ of multiplicity at least $1$. According to Lemma $\ref{similar_spectra}$, $\lambda_*$ is in the spectrum of $D R_{EKW}[F^*]$, and according to Lemma $\ref{same_spectra}$, $\lambda_* \in D \cR_{EKW}[s^{*}]$.

\begin{prop} \label{R_EKW_subset}
Suppose that  $\beta$, $\rho$, $\varrho$ and the neighborhood  $\cB_\varrho(s^{*}) \subset \cA_s^\beta(\rho)$  satisfy the hypothesis of Lemma $\ref{same_spectra_new}$. Furthermore, suppose that the operator $\cR$ is analytic and compact in $\cB_\varrho(s^{*})$. 

Then
$${\rm spec}(D \cR_{EKW}[s^{*}]) \setminus \{ \lambda_* \} \subset {\rm spec}(D \cR[s^{*}]),$$
and $\psi^{EKW}_{s^*}$ is an eigenvector of   $D \cR[s^{*}]$ associated with the eigenvalue $0$.

In addition,
$$
{\rm spec}(D \cR[s^{*}]) \subset {\rm spec}(D \cR_{EKW}[s^{*}]),$$
and if $\lambda_* \notin {\rm spec}(D \cR[s^{*}])$, then $\lambda_*$ has multiplicity $1$ in  ${\rm spec}(D \cR_{EKW}[s^{*}])$.

\end{prop}

\begin{proof}
First, notice the difference between the definition of $\lambda$ in $(\ref{EKW_def})$ 
$$s(G(\lambda,0))=0,$$
and in Definition $(\ref{strong_renorm})$
$$s(G(\lambda+t \lambda^2,0))=0$$
(we will use the notation  $\lambda_{EKW}$ below to emphasize the difference). This implies that if $D_s \lambda_{EKW}[s] \psi$ is an action of the derivative of $\lambda_{EKW}[s]$ on a vector $\psi$, then
$$D_s \lambda[s^*] \psi=D_s \lambda_{EKW}[s^*] \psi -\lambda_*^2 D_st[s^*] \psi$$
is that of the derivative of $\lambda[s]$.

Similarly,
\begin{eqnarray}
\nonumber D_s \mu_{EKW}[s^*] \psi &=&  \left[ \partial_1 (s^* \circ G)(\lambda_*,0)+ \lambda_* \partial^2_{1} (s^* \circ G) (\lambda_*,0)\right]D_s \lambda_{EKW}[s^*] \psi \\
\nonumber &\phantom{=}&\phantom{ \left.  \partial_1 (S^* \circ G)(\lambda_*,0) \right . }+ \lambda_* \partial_1 ( D_s\cP_{EKW}[s^*] \psi )(\lambda_*,0),\\
\nonumber D_s \mu[s^*] \psi &=&\left[ \partial_1 (s^* \circ G)(\lambda_*,0)+ \lambda_* \partial^2_{1} (s^* \circ G) (\lambda_*,0)\right] D_s \lambda[s^*] \psi  \\
\nonumber &\phantom{=}&\phantom{ \left.  \partial_1 (S^* \circ G)(\lambda_*,0) \right. }  + \lambda_* \partial_1 ( D_s\cP_{EKW}[s^*]\psi)(\lambda_*,0)\\
\nonumber &\phantom{=}&\phantom{ \left.  \partial_1 (S^* \circ G)(\lambda_*,0) \right. } + \lambda_*^3 \partial_1^2 (s^* \circ G)(\lambda_*,0) D_st[s^*] \psi\\
\nonumber &=& D_s \mu_{EKW}[s^*] \psi- \partial_1\cP_{EKW}[s^*](\lambda_*,0)  \lambda_*^2  D_st[s^*]\psi \\
\nonumber &=& D_s \mu_{EKW}[s^*] \psi-  \lambda_* \mu_*  D_st[s^*]\psi.
\end{eqnarray}

Therefore,
\begin{eqnarray}\label{evec_R_REKW}
\nonumber D \cR[s^{*}] \psi=D \cR_{EKW}[s^{*}]\psi &+& 2 \lambda_* \left( D_s t[s^*] \psi \right) s^* \pi_y+{1 \over \mu_*} \left( D \cP_{EKW}[s^*] \cdot  \left( D_s\xi_t \psi \right) \right)\circ \lambda_*\\
\nonumber  &-& D_s t[s^*] \psi {\lambda^2_* \over \mu_*} D \cP_{EKW}[s^*] \circ \lambda_*  \cdot (\pi_x,\pi_y) \\
\nonumber & +&\lambda_*  D_st[s^*]\psi  s^*\\
\nonumber =D \cR_{EKW}[s^{*}]\psi & -& \lambda_* \left(D_s t[s^*] \psi\right)  D s^*  \cdot (\pi_x,\pi_y) +\lambda_* \left( D_st[s^*]\psi\right)  s^*\\
\nonumber &+&  \lambda_* \left( D_s t[s^*] \psi \right) \psi_{s^*}\\
=D \cR_{EKW}[s^{*}]\psi &+& \lambda_* \left( D_s t[s^*] \psi \right)  \psi^{EKW}_{s^*}\label{R_EKW_R}
\end{eqnarray}
where
\begin{eqnarray}
\nonumber D_s t[s^*] \psi&=&  -\lambda_*^{-1} \|\psi^{EKW}_{s^*}\|_\rho^{-1} \| E(\lambda_*)  \left(D \cR_{EKW} [s^{*}] \psi \right)\|_\rho,\\
\nonumber D_s \xi_t[s^*] \psi(x,y)&=&\left(D_s t \psi \right) (x^2 ,y^2) \\
\nonumber &=&  -\lambda_*^{-1} \|\psi^{EKW}_{s^*}\|_\rho^{-1} \| E(\lambda_*)  \left(D \cR_{EKW} [s^{*}] \psi \right)\|_\rho (x^2 ,y^2).
\end{eqnarray}

Similarly to Lemma $(\ref{elimination_lemma})$, we get that if $\psi$ is an eigenvector of $D \cR_{EKW}[s^*]$ associated with the eigenvalue $\delta \ne \lambda_*$, then  $\psi \ne  \psi^{EKW}_{s^*}$, and
$$E(\lambda_*)  \left(D \cR_{EKW} [s^{*}] \psi\right)=\delta E(\lambda_*) \psi=0,$$ 
so is  $D_s t[s^*] \psi$, and  
$$D \cR[s^{*}] \psi=D \cR_{EKW}[s^{*}]\psi=\delta \psi.$$

If $\delta=\lambda_*$ and $\psi=\psi^{EKW}_{s^*}$, then 
$$D_s t[s^*] \psi=  -1,\quad D_s \xi_t[s^*] \psi(x,y)=- (x^2 ,y^2),$$
and therefore,
$$D \cR[s^{*}] \psi^{EKW}_{s^*}=\lambda_* \psi^{EKW}_{s^*}- \lambda_*\psi^{EKW}_{s^*}=0,
$$
and $\psi^{EKW}_{s^*}$ is an eigenvector of $D  \cR[s^*]$ associated with the eigenvalue $0$.


Vice verse, by (\ref{evec_R_REKW}), if $\psi$ is an eigenvector of $D\cR[s^*]$ associated with the eigenvalue $\delta\neq \lambda_*$, then
\begin{eqnarray*}
D \cR_{EKW}[s^{*}] (\psi+a\psi^{EKW}_{s^*})  &=& D \cR[s^{*}]\psi-\lambda_*(D_s t[s^*] (\psi+a \psi^{EKW}_{s^*})\psi^{EKW}_{s^*}\\
&=&\delta \psi-\lambda_*(D_s t[s^*] \psi-a)\psi^{EKW}_{s^*}
\end{eqnarray*}
Hence, $\psi+\frac{\lambda_*D_s t[s^*] \psi}{\lambda_*-\delta}\psi^{EKW}_{s^*}$ is an eigenvector of $D\cR_{EKW}[s^*]$ with the eigenvalue $\delta$. 

Finally, assume that $\lambda_* \notin {\rm spec}(D \cR[s^{*}])$, but that there exists an eigenvector $\varphi \neq \psi_{s^*}^{EKW}$ of $D\cR_{EKW}[s^*]$ with eigenvalue $\lambda_*$. Then 
$$ D_s t[s^*] \varphi = -\frac{\|\varphi\|_\rho}{\|\psi^{EKW}_{s^*}\|_\rho},$$
and, by (\ref{evec_R_REKW}),
\begin{eqnarray*}
D \cR[s^{*}]\left(\varphi-\frac{\|\varphi\|_\rho}{\|\psi^{EKW}_{s^*}\|_\rho}\psi^{EKW}_{s^*}\right) & = &
D \cR[s^{*}]\varphi \\
& = &
\lambda_*\varphi+\lambda_*\left(-\frac{\|\varphi\|_\rho}{\|\psi^{EKW}_{s^*}\|_\rho}\right)\psi^{EKW}_{s^*} \\
& = &
\lambda_*\left(\varphi-\frac{\|\varphi\|_\rho}{\|\psi^{EKW}_{s^*}\|_\rho}\psi^{EKW}_{s^*}\right).
\end{eqnarray*}
This contradiction finishes the proof.
\end{proof}

So far we were not able to make any claims about the multiplicity of the eigenvalue $\lambda_*$ in the spectrum of $D \cR_{EKW}[s^{*}]$. However, we will demonstrate in Section $\ref{R_spectral_prop}$ that it is indeed equal to $1$.

\begin{definition}\label{strong_Frenorm}
Set, formally,
\begin{eqnarray}
\label{Fren_eq} R[F]&=&\Lambda^{-1}_F \circ P[F] \circ \Lambda_F,\\
\nonumber P[F]&=& S_{t[F]}^{-1} \circ F \circ F \circ S_{t[F]},
\end{eqnarray}
where $S_{t[F]}$  is as in $(\ref{coord})$, $\Lambda_F(x,u)=(\lambda[F] x, \mu[F] u)$, 
$$t[F]=-{1 \over \lambda_* \|h_{F^*,\sigma}\|_\cD} \|E(\lambda_*)(R_{EKW}[F]-F^*)\|_{\cD},$$
where
$$\sigma=\sigma^1_{1,0}-2 \sigma^2_{1,0}-\sigma^1_{0,0}+\sigma^2_{0,0},$$
and, furthermore,
\begin{eqnarray}
\nonumber \lambda[F]&=&\pi_x P[F](0,0), \\
\nonumber \mu[F]&=&{-\lambda[F] \over  \pi_x P[F]_2(0,0)}.
\end{eqnarray}

\end{definition}

The above is a formal definition. As usual, one would have to verify the properties of being well-defined, analytic and compact, in a setting of a specific functional space.




\bigskip

\section{Spectral properties of $\cR$. Proof of Main Theorem} \label{R_spectral_prop}

We will now describe our computer-assisted proof of Main Theorem.

To implement the operator $D \cR[s^*]$ on the computer, we would have to implement  the Riesz projection as  well.  Unfortunately, this is not easy, therefore, we  do it only approximately, using the operator $\cR_c$ introduced in the Definition $\ref{simple_strong_renorm}$.  Specifically, the component $(0,3)$ of the composition $s \circ G$ will be consistently normalized to be
$$c_0= \left(s_0 \circ G[s_0] \right)_{(0,3)},$$
where $s_0$ is our polynomial approximation for the fixed point. 

The operator $\cR_c$ differs from  $\cR$ (cf.$\ref{strong_renorm}$) only in the ``amount'' by which the eigendirection $\psi^{EKW}_{s^*}$ is ``eliminated''. In particular, as the next proposition demonstrates,  $R_c$ is still analytic and compact in the same neighborhood of $s_0$.

\bigskip

\begin{prop}\label{tilde_R_analyticity}
There exists a polynomial $s_0  \subset \cB_{R}(s^0) \subset \cA_s(1.75)$, where $R$ and $s^0$ are as in Theorem $\ref{MyTheorem}$, such that the operators $\cR_c$, $c \in [c_0-\delta,c_0+\delta]$,
$$c_0=\left(s_0 \circ G[s_0] \right)_{(0,3)} \quad {\rm and} \quad \delta=1.068115234375 \times 10^{-4},$$ 
are well-defined and   analytic as  maps from $\cB_{\varrho_0}(s_0)$, $ \varrho_0=5.79833984375 \times 10^{-4}$, to $\cA_s(1.75)$.

Furthermore, the operators $\cR_c$ are compact in $B_{R}(s^0) \subset \cA(\rho)$, with $\cR_c[s] \in \cA(\rho')$, $\rho'=1.0699996948242188 \rho$.
\end{prop}

\begin{proof}
The proof is almost identical to that of Proposition $\ref{R_analyticity}$, with a different (but still sufficiently small) bound on $|t_c[s]|$.
\end{proof}

\bigskip

The following Lemma shows that the spectra of the operators $\cR$ and  $\cR_{c}$  are close to each other.

\begin{lemma}\label{R_containment}
Suppose that the neighborhood  $\cB_{\varrho_0}(s_0)$,  with $\varrho_0$ as in  Propositions $\ref{R_analyticity}$ and $\ref{tilde_R_analyticity}$, contains a fixed point $s^*$ of $\cR$, and of  $\cR_{c^*}$ for 
$$c^*= \left(s^* \circ G[s^*] \right)_{(0,3)}.$$ 
Set
$$\delta=0.00124359130859375,$$
then
$${\rm spec}\left(D \cR[s^*]\right) \setminus \left\{z \in \field{C}: |z| \le \delta \right\} \subset {\rm spec}\left(D \cR_{c^*}[s^*]\right) \setminus \left\{z \in \field{C}: |z| \le \delta \right\}. $$
\end{lemma}

\begin{proof}
According to  Propositions $\ref{R_analyticity}$ and $\ref{tilde_R_analyticity}$, under the hypothesis of the Lemma, $\cR$ and $\cR_{c^*}$ are analytic and compact as operators from  $\cB_{\delta}(s_0)$ to $\cA_s(1.75)$.

Recall, that $\psi^{EKW}_{s^*}$ is an eigenvector of $D \cR_{EKW}[s^*]$ corresponding to the eigenvalue $\lambda_*$. 


We consider the action of $D \cR_{c^*}[s^*]$ on a vector $\psi$. Similarly to $(\ref{R_EKW_R})$,
\begin{eqnarray}
\nonumber D \cR_{c^*}[s^*] \psi&=&D \cR_{EKW}[s^{*}]\psi +  \lambda_* \left( D_s t_c[s^*] \psi \right) \psi_{s^*} +\lambda_*  \left( D_s t_c[s^*] \psi \right) \tilde{\psi}\\
\nonumber &=& D \cR[s^{*}] \psi +  \lambda_*  \left( \left( D_s t_c[s^*] -D_s t[s] \right) \psi \right) \psi^{EKW}_{s^*}.
\end{eqnarray}

Now, let $\psi$ be an eigenvector of $D  \cR[s^*]$ of eigenvalue $\kappa  \ne 0$ (that is, $\psi\ne\psi_{s^*}^{EKW}$). Consider the action of $D \cR_{c^*}[s^*]$  on $\psi +a \psi_{s^*}^{EKW}$.
$$D \cR_{c^*} [s^*]  (\psi+a \psi_{s^*}^{EKW})  = \kappa \psi+ \lambda_* \left( D_s t_c[s^*]-D_s t[s^*] \right) (\psi+a \psi_{s^*}^{EKW})   \psi_{s^*}^{EKW}.$$

Notice,
\begin{eqnarray*}
D_s t_c[s^*]\psi_{s^*}^{EKW}&=& D_s t_c[s^*](\psi_{s^*}+\psi_u-\psi_x)\\
&=&
-{1\over 4} {\left( D \cP_{EKW}[s^*] ( \psi_{s^*}+\psi_u-\psi_x ) \right)_{0,3}   \over \cP_{EKW}[s^*]_{0,2} }\\
& &-
{1\over 4} {\left( D \cP_{EKW}[s^*] ( \psi_{s^*}+\psi_u-\psi_x ) \right)_{0,2} (c-\cP_{EKW}[s^*]_{0,3})  \over \left(\cP_{EKW}[s^*]_{0,2}\right)^2 } \\
&=&
-{1\over 4} { \left( \psi_{\cP_{EKW}[s^*]}+\cP_{EKW}[s^*] -D \cP_{EKW}[s^*] \cdot (\pi_x,\pi_y) \right)_{0,3}   \over \cP_{EKW}[s^*]_{0,2} }\\
& &-
{1\over 4} { \left( \psi_{\cP_{EKW}[s^*]}+\cP_{EKW}[s^*] -D \cP_{EKW}[s^*] \cdot (\pi_x,\pi_y) \right)_{0,2} (c-\cP_{EKW}[s^*]_{0,3})  \over \left(\cP_{EKW}[s^*]_{0,2}\right)^2 }\\
&=&
-{1\over 4} { \left( \partial_2\cP_{EKW}[s^*]\right)_{0,1}+2 \left(\cP_{EKW}[s^*]\right)_{0,2}   \over \cP_{EKW}[s^*]_{0,2} }\\
& &
-{1\over 4} { \left( \cP_{EKW}[s^*] \right)_{0,3}  - \left( \partial_2 \cP_{EKW}[s^*] \right)_{0,2}   \over \cP_{EKW}[s^*]_{0,2} }\\
& &-
{1\over 4} {\left( \left( \partial_2\cP_{EKW}[s^*]\right)_{0,0}+2 \left(\cP_{EKW}[s^*]\right)_{0,1}\right)(c-\cP_{EKW}[s^*]_{0,3})   \over \left(\cP_{EKW}[s^*]_{0,2} \right)^2}\\
& &-
{1\over 4} { \left(\left( \cP_{EKW}[s^*] \right)_{0,2}  - \left( \partial_2 \cP_{EKW}[s^*] \right)_{0,1}\right)(c-\cP_{EKW}[s^*]_{0,3})   \over \left(\cP_{EKW}[s^*]_{0,2}\right)^2 }\\
&=&
-1+{1\over 2} { c^*  \over \cP_{EKW}[s^*]_{0,2} }-{1\over 4}\left({3 {\cP_{EKW}[s^*]_{0,1}} \over {\cP_{EKW}[s^*]_{0,2}}} -1 \right) {c-c^* \over \cP_{EKW}[s^*]_{0,2}}\\
&=&
-1+C,\\
D_s t[s^*]\psi_{s^*}^{EKW}
&=&
-1
\end{eqnarray*}

Denote $d_1 \equiv D_s t_c[s^*] \psi$ and $d_2 \equiv  D_s t[s^*] \psi$, then
\begin{eqnarray}
\nonumber D \cR_{c^*} [s^*]  (\psi+a \psi_{s^*}^{EKW})  &=& \kappa \psi+ \lambda_* ( d_1-d_2+a(-1+C)+a) \psi_{s^*}^{EKW}\\
\nonumber  &=& \kappa \left( \psi+ {\lambda_*\over \kappa} ( d_1-d_2+aC) \psi_{s^*}^{EKW} \right),
\end{eqnarray}
and we see that the equation 
$$a={\lambda_*\over \kappa} ( d_1-d_2+aC)$$
has  a unique solution $a$ if 
\begin{equation}\label{kappa}
\kappa \ne \lambda_* C.
\end{equation}

For such $\kappa$, the vector
$$\psi+{\lambda_*(d_1-d_2)\over \kappa-\lambda_* C }  \psi_{s^*}^{EKW}$$
is an eigenvector of $D \cR_{c^*}[s^*]$ associated with the eigenvalue $\kappa$.

The eigenvalues $\kappa$ as in $(\ref{kappa})$ satisfy
$$|\kappa| > 0.00124359130859375$$ 


\end{proof}

We will now describe a rigorous computer upper bound on the spectrum of the operator $D R_c[s^*]$. 

\bigskip

\noindent {\it Proof of part ii) of Main Theorem.}

\medskip

\noindent {\it Step 1).} 
Recall the Definition  $\ref{B_space}$ of the Banach subspace $\cA_s(\rho)$ of $\cA(\rho)$. We will now choose a new bases $\{\psi_{i,j}\}$ in $\cA_s(\rho)$. Given $s \in \cA_s(\rho)$ we write its Taylor expansion in the form
$$s(x,y)=\sum_{(i,j)\in I} s_{i,j} \psi_{i,j}(x,y),$$ 
where $\psi_{i,j} \in \cA_s(\rho)$:
\begin{eqnarray}
\nonumber \tilde{\psi}_{i,j}(x,y)&=&x^{i+1} y^j,  \quad i=-1, \quad j \ge 0, \\
\nonumber \tilde{\psi}_{i,j}(x,y)&=& x^{i+1} y^j +{i+1 \over {j+1}}  x^{j+1}  y^i,  \quad i >-1, \quad j \ge i,\\
\nonumber \psi_{i,j}&=&{\tilde{\psi}_{i,j} \over \|\tilde{\psi}_{i,j} \|_\rho}, \quad i \ge -1, \quad j \ge \max\{0,i\},
\end{eqnarray}
and the index set $I$ of these basis vectors is defined as
$$I=\{(i,j) \in \field{Z}^2: \quad i \ge -1, \quad j \ge \max\{0,i\}  \}.$$

Denote $\tilde{\cA}_s(\rho)$ the set of all sequences 
$$\tilde{s}=\left\{s_{i,j}: s_{i,j} \in \field{C}, \sum_{(i,j) \in I} |s_{i,j}|  < \infty \right\}.$$

Equipped with the $l_1$-norm 
\begin{equation}\label{l_1}
|s|_1=\sum_{(i,j) \in I} |s_{i,j}|,
\end{equation}
$\tilde{\cA}_s(\rho)$ is a Banach space, which is isomorphic to $\cA_s(\rho)$. Clearly, the isomorphism $J: \cA_s(\rho) \mapsto \tilde{\cA_s}(\rho)$ is an isometry:
$$\| \cdot \|_\rho=| \cdot |_1.$$

We divide the set $I$  in three disjoint  parts: 
\begin{eqnarray}
\nonumber I_1 &=& \{(i,j) \in I: i+j<N \},\\
\nonumber  I_2 &=& \{(i,j) \in I:  N  \le i+j < M \},\\
\nonumber I_3 &=& \{(i,j) \in I:  i+j \ge M \},
\end{eqnarray} 
with
$$N=22, \quad M=60.$$
We will denote the cardinality of the first set as $D(N)$, the cardinality of $I_1 \cup I_2$ as $D(M)$.

We assign a single index to vectors $\psi_{i,j}$, $(i,j) \in I_1 \cup I_2$, as follows:
\begin{eqnarray}
\nonumber &&k(-1,0)=1,\quad k(-1,1)=2, \quad  \ldots ,\quad  k(-1,M)=M+1, \quad k(0,0)=M+2, \\
\nonumber &&  k(0,1)=M+3,\quad \ldots, \quad k\left(\left[{M-1\over 2} \right], M-1-\left[{M-1\over 2} \right]\right)=D(M).
\end{eqnarray}
This correspondence $(i,j)  \mapsto k$ is one-to-one, we will, therefore, also use the notation $(i(k),j(k))$. 


For any $s \in \cA_s(\rho)$, we define the following projections on the subspaces of the linear subspace $E_{D(N)}$ spanned by $\{\psi_k\}_{k=1}^{D(N)}$.
$$\Pi_k s=s_{i(k),j(k)} \psi_k, \quad \Pi_{E_{D(N)}} s = \sum_{m \le {D(N)}}  \Pi_m s.$$

Fix 
$$c_0=(s_0 \circ G[s_0])_{0,3},$$
where $s_0$ is some good numerical approximation of the fixed point. Denote for brevity $\cL_{c}^s \equiv D \cR_{c}[s]$. We can now write a matrix representation of the finite-dimensional linear operator 
$$\Pi_{E_{D(N)}}  \cL_{c_0}^{s_0}  \Pi_{E_{D(N)}}$$
as 
$$ D_{n,m}= \Pi_m \cL_{c_0}^{s_0} \psi_n .$$

\medskip

\noindent {\it Step 2).}  We compute  the unit eigenvectors $e_k$ of the matrix $D$ {\it numerically}, and form a $D(N)\times D(N)$ matrix $A$ whose columns are the approximate eigenvectors $e_k$. We would now like to find a rigorous bound $\bB$ on the inverse $B$ of $A$.

Let $B_0$ be an approximate inverse of $A$. Consider the operator $C$ in the Banach space of all $D(N) \times D(N)$ matrices (isomorphic to $\fR^{D(N)^2}$) equipped with the $l_1$-norm, given by
$$C[B]=(A+\fI)B-\fI.$$
Notice, that if $B$ is a fixed point of $C$ then $A  B=\fI$. Consider a ``Newton map'' for $C$:
$$N[z]=z+C[B_0-B_0z]-B_0+B_0z.$$
If $z$ is a fixed point of $N$, then $B_0-B_0z$ is a fixed point of $C$. Furthermore,
$$DN[z]=\fI-A B_0$$
is constant. We therefore, estimate $l_{\infty}$ matrix norms
$$\|N[0]\|_1 \le \equiv \eps, \quad \|\fI-A B_0\|_1 \le \equiv \cD,$$
and obtain via the Contraction Mapping Principle, that the inverse of $A$ is contained in the $l_1$ \ $\delta$-neighborhood of $B_0$, with
$$\delta=\|B_0\|_1 {\eps \over 1-\cD}.$$

\medskip

\noindent {\it Step 3).}  Define the linear operator 
$$\cA=A \Pi_{E_{D(N)}} \bigoplus \left(\fI- \Pi_{E_{D(N)}}\right),$$
and its inverse
$$\cB=B \Pi_{E_{D(N)}} \bigoplus \left(\fI- \Pi_{E_{D(N)}}\right).$$
Consider the action of the operator $\cL_{c_0}^s$ in the new basis
$$e_k={\tilde{e}_k \over \| \tilde{e}_k \|_\rho}, 1 \le k \le D(N), \quad e_k \equiv \psi_k, \quad k > D(N),$$
where
\begin{equation}\label{basis}
[e_1, e_2, \ldots, e_{D(N)}] \equiv [\psi_1, \psi_2, \ldots,\psi_{D(N)}] A,
\end{equation}
in $\cA_s(\rho)$. To be specific, we consider a new Banach space $\hat{\cA}_s(\rho)$: the space of all  functions 
$$s=\sum_k c_k e_k,$$
analytic on a bi-disk $\cD_\rho$, for  which the norm
$$\|s\|_1=\sum_k |c_k|$$
is finite. 

For any $s \in \hat{\cA}_s(\rho)$, we define the following projections on the basis vectors.
$$P_i s=c_i  e_i, \quad P_{>k} s=\left(\fI-\sum_{i=1}^{k} P_i \right) s.$$

Clearly, the Banach spaces $\cA_s(\rho)$ and $\hat{\cA}_s(\rho)$ are isomorphic, while the norms $\| \cdot \|_\rho$ and $\| \cdot \|_1$ are equivalent. We can use $(\ref{basis})$ to compute the equivalence constant $\alpha$ in 
$$\alpha \|  \cdot \|_1  \ge \| \cdot \|_\rho = |  \cdot |_1$$
(recall, norms $\| \cdot \|_\rho$ and $| \cdot |_1$, defined in $(\ref{l_1})$ are equal). 
Notice, that
\begin{eqnarray}
\nonumber s&=&\sum_k c_k e_k= \sum_{1 \le k \le D(N)} c_k \left( \sum_{1 \le i \le D(N)} A^i_k \psi_i \right) +\sum_{k>D(N)} c_k \psi_k \\
\nonumber &=& \sum_{1 \le i \le D(N)} \left(\sum_{1 \le k \le D(N)} c_k A^i_k \right) \psi_i+\sum_{i>D(N)} c_i \psi_i ,
\end{eqnarray}
therefore, if $A^i$ is the $i$-th row of the matrix $A$, then
\begin{eqnarray}
\nonumber |s|_1&=&\sum_{1 \le i \le D(N)}  \left| \sum_{1  \le k  \le D(N)} c_k A^i_k \right| + \sum_{i>D(N)} |c_i| \\
\nonumber &\le&  \sum_{1 \le i \le D(N)} \left(  \|A^i\|_\infty \sum_{1 \le k\le D(N)} |c_k|\right)+\sum_{i>D(N)} |c_i|\\
\nonumber &=& \left[\sum_{1 \le i \le D(N)}  \|A^i\|_\infty \right] \sum_{1 \le k\le D(N)} |c_k|+\sum_{i>D(N)} |c_i| \\
\nonumber &\le& \max\left\{ \sum_{1 \le i \le D(N)}  \|A^i\|_\infty,1\right\}\|s\|_1 
\end{eqnarray}
and 
$$
\alpha=\max\left\{\sum_{ 1 \le i \le D(N)} \| A^i\|_\infty,1 \right\}.
$$

The constant has been rigorously evaluated on the computer:
\begin{equation}\label{eq_const}
\alpha  \le 49.435546875.
\end{equation}

The operator $\cL_{c_0}^s$  is ``almost'' diagonal in this new basis for all $s \in \cB_\varrho(s_0) \subset \cA_s(\rho)$,
$$\varrho = 6.0 \times 10^{-12}.$$

 We proceed to quantify this claim.

\bigskip
\begin{tabular}{l l l l}
    \!\!\!\!\!\!\!\!\!\!\! $\| P_2 \cL_{c_0}^s e_1 \|_1$ &  \!\!\!\!\!\!\! $\le  5.19007444381714 \times 10^{-4}$ , &  \!\!\!\!\!\!\! $\| P_1 \cL_{c_0}^s e_2 \|_1$ &   \!\!\!\!\!\!\!  $\le 1.76560133695602 \times 10^{-4}$,\\
   \!\!\!\!\!\!\!\!\!\!\! $\| P_{>2}  \cL_{c_0}^s e_1 \|_1$ &   \!\!\!\!\!\!\! $\le 3.5819411277771 \times 10^{-3}$,&  \!\!\!\!\!\!\! $\| P_{>2}  \cL_{c_0}^s e_2 \|_1$ &    \!\!\!\!\!\!\!  $\le 1.49521231651306 \times 10^{-3}$,\\
    \!\!\!\!\!\!\!\!\!\!\! $\| P_1  \cL_{c_0}^s P_{>2} \|_1$ &  \!\!\!\!\!\!\! $\le 1.22539699077606 \times 10^{-4}$,&  \!\!\!\!\!\!\! $\| P_2  \cL_{c_0}^s P_{>2} \|_1$ & \!\!\!\!\!\!\!  $\le 8.23289155960083 10^{-5}$,
\end{tabular}

\bigskip
\noindent for all $h \in \cB_\varrho(s_0) \subset \cA_s(\rho)$.

\medskip

\noindent {\it Step 4).} We will  now demonstrate existence of a fixed point $s^*_{c_0}$  in $\cB_\varrho \in \cA_s(\rho)$, of the operator $\cR_{c_0}$, where
$$c_0=(s_0 \circ G[s_0])_{0,3}.$$

We will use the Contraction Mapping  Principle in the following form. Define the following linear operator on $\cA_s(\rho)$
$$M  \equiv \left[\fI - K \right]^{-1},$$
where
$$K h \equiv \hat{\delta}_1 P_1 h + \hat{\delta}_2 P_2 h,$$
and $\hat{\delta}_1$ and $\hat{\delta}_2$ are defined via 
$$P_1 \cL_{c_0}^{s_0} e_1=\hat{\delta}_1 e_1, \quad P_2 \cL_{c_0}^{s_0} e_2=\hat{\delta}_2 e_2.$$
Consider the operator  
$$\cN[h]=h+\cR_{c_0}[s_0+M h]-(s_0+M h)$$
on $\hat{\cA}_s(\rho)$ and for all $z$. 

The operator $\cN$ is analytic and compact on $\cB_{ \|M\|^{-1}_1 \alpha^{-1}  \varrho}(0)$, where $c$ is the norm equivalence constant $(\ref{eq_const})$, and
$$\|M\|_1=\max \left\{\left| {1 \over 1-\hat{\delta}_1} \right|, \left| {1 \over 1-\hat{\delta}_2} \right|,1 \right\}=1.$$

Notice, that if $h^*$ is a fixed point of $\cN$, then $s_0+M h^*$ is a fixed point of $R_{c_0}$.

The derivative norm of the operator $\cN$ is ``small'', indeed,
\begin{eqnarray}
\nonumber D \cN[h] &=& \field{I} +D \cR_{c_0}[s_0+M h] \cdot M -M \\
\nonumber &=& \left[ M^{-1} +  D  \cR_{c_0} [s_0+M h]  -  \field{I} \right] \cdot M \\
\nonumber &=& \left[\field{I}-K  +D \cR_{c_0}[s_0+M h] -   \field{I} \right] \cdot M \\
\nonumber &=& \left[ D \cR_{c_0}[s_0+M h] - K \right] \cdot M.
\end{eqnarray}

We have evaluated the operator norm of this derivative for all  $h \in \cB_{ \alpha^{-1}  \varrho}(0)$:
$$ \| D \cN[h] \|_1  \equiv \cD \le 0.1258544921875$$

At the same time
$$\|\cN[0]\|_1 = \|\cR_{c_0}[s_0]-s_0\|_1 \equiv \epsilon   \le 4.9560546875 \times 10^{-16}.$$

We can now  see that the hypothesis of the Contraction Mapping Principle is indeed verified:
$$\epsilon < 4.9560546875 \times 10^{-14} < 1.058349609375 \times 10^{-13} < (1-\cD) \alpha^{-1} \varrho,$$
and therefore,  the neighborhood $\cB_{\epsilon/(1-\cD) }(0) \subset \cB_{0.5 \alpha^{-1} \varrho} (0)$ contains a fixed point $h^*$ of $\cN$, i.e. the neighborhood $\cB_{\varrho/2} (s_0)\subset \cB_\varrho(s_0) \subset \cA_s(\rho)$ contains a fixed point $s^*_{c_0}=s_0+M h^*$ of $\cR_{c_0}$.

We quote here for reference purposes the bounds on the values of the scalings $\lambda[s^*_c]$ and $\mu[s^*_c]$:
\begin{eqnarray}
\label{lambda_b} \lambda[s^*_c]&=&[-0.248875288734817765,-0.248875288702286711], \\
\label{mu_b} \mu[s^*_c]&=&[ 0.0611101382055370338, 0.0611101382190655586].
\end{eqnarray}

\medskip

\noindent {\it Step 5).}  Notice, that in general,
$$\left(s^*_{c_0} \circ G[s^*_{c_0}]\right)_{0,3} \ne c,$$ 
therefore
$$t_{c_0}[s^*_{c_0}] \ne 0.$$

However, $t_{c_0}[s^*_{c_0}]$ is a small number which we have estimated to be
\begin{equation}\label{t_c}
|t_{c_0}[s^*_{c_0}]| < 7.89560771750566329 \times 10^{-12}.
\end{equation}

Consider the map $F^*_{c_0}$ generated by $s_{c_0}^*$. Recall that by Theorem $\ref{fp_properties}$, there exists a simply connected open set $\cD$ such that $F^*_{c_0} \in \cO_2(\cD)$. The fixed point equation for the map $F^*_{c_0}$ is  as follows:
$$\Lambda_{F^*_{c_0}}^{-1} \circ S_{t_{c_0}[s^*_{c_0}]}^{-1} \circ F^*_{c_0} \circ F^*_{c_0} \circ S_{t_{c_0}[s^*_{c_0}]} \circ \Lambda_{F^*_{c_0}}=F^*_{c_0}.$$

\begin{flushright}$\Box$\end{flushright}

\bigskip


\end{document}